\documentclass[a4paper,12pt]{article}
\usepackage{amsmath,amsthm,amssymb,amsfonts}
\usepackage{bbm,bm}
\usepackage{latexsym}
\usepackage{mathrsfs}
\usepackage{threeparttable}
\usepackage{tabularx}
\usepackage{booktabs}
\usepackage{graphicx}
\usepackage{epstopdf}
\usepackage{color}
\usepackage{indentfirst}
\usepackage{mathrsfs}
\usepackage{geometry}
\usepackage{caption}
\usepackage{lineno}
\usepackage{cleveref}
\usepackage{abstract}
\usepackage{enumerate,mdwlist}
\usepackage{cases}
\usepackage{subfigure}

\usepackage[numbers,sort&compress]{natbib}

\theoremstyle{plain}
\newtheorem{The}{Theorem}[section]    
\newtheorem{Lem}[The]{Lemma}

\theoremstyle{definition}
\newtheorem{Def}{Definition}

\numberwithin{equation}{section} 

\makeatletter

\newcommand{\Rmnum}[1]{\expandafter\@slowromancap\romannumeral #1@}
\makeatother
\title{The structure connectivity of Data Center Networks }

\author{Lina Ba and Heping Zhang\thanks{Corresponding author.}}
\date{{\small School of Mathematics and Statistics, Lanzhou University,
 Lanzhou, Gansu 730000, P.R. China}\\
{\small E-mails:\ baln19@lzu.edu.cn, zhanghp@lzu.edu.cn}}

\begin{document}

\maketitle
\begin{abstract}
 Last decade, numerous giant data center networks are built
to provide increasingly fashionable web applications. For two integers $m\geq 0$ and $n\geq 2$, the $m$-dimensional DCell network with $n$-port switches $D_{m,n}$ and $n$-dimensional BCDC network $B_{n}$ have been proposed.
 Connectivity is a basic parameter to measure fault-tolerance of networks.
 As generalizations of connectivity, structure (substructure) connectivity was recently proposed. Let $G$ and $H$ be two connected graphs. Let $\mathcal{F}$ be a set whose elements are subgraphs of $G$, and every member of $\mathcal{F}$ is isomorphic to $H$ (resp. a connected subgraph of $H$). Then $H$-structure connectivity $\kappa(G; H)$ (resp. $H$-substructure connectivity $\kappa^{s}(G; H)$) of $G$ is the size of a smallest set of $\mathcal{F}$  such that  the rest of $G$ is disconnected or the singleton when removing $\mathcal{F}$. Then it is meaningful to calculate the structure connectivity of data center networks on some common structures, such as star $K_{1,t}$, path $P_k$, cycle $C_k$, complete graph $K_s$ and so on. In this paper, 
we obtain that $\kappa (D_{m,n}; K_{1,t})=\kappa^s (D_{m,n}; K_{1,t})=\lceil \frac{n-1}{1+t}\rceil+m$ for $1\leq t\leq m+n-2$ and $\kappa (D_{m,n}; K_s)= \lceil\frac{n-1}{s}\rceil+m$ for $3\leq s\leq n-1$ by analyzing the structural properties of $D_{m,n}$. We also compute $\kappa(B_n; H)$ and $\kappa^s(B_n; H)$ for $H\in \{K_{1,t}, P_{k}, C_{k}|1\leq t\leq 2n-3, 6\leq k\leq 2n-1 \}$ and $n\geq 5$ by using $g$-extra connectivity of $B_n$.

     \vskip 0.1in


    \vskip 0.1 in


    \medskip
\end{abstract}
\section{Introduction}
Networks are often modeled by connected graphs. A node replaces with a vertex and a link replaces with an edge.
Let $G$ be a graph. We set $V(G)$ as vertex set of $G$ and $E(G)$ as edge set of $G$.
Connectivity is a basic parameter in measuring fault tolerance of networks.
In order to more accurately measure fault-tolerance of large-scale parallel processing systems,
$\rho$-conditional connectivity for some property $\rho$ proposed by Harary \cite{1} referred to the minimum cardinality of a vertex subset $S$ of $G$ such that $G-S$ is disconnect and every component has property $\rho$. Thus, $h$-extra connectivity $\kappa_h(G)$ is introduced by F\`{a}brega and Fiol \cite{2}, and the definition emphasizes that the every component left behind has at least $h+1$ vertices;  $h$-restricted connectivity $\kappa ^h(G)$ is came up with by Esfahanian and Hakimi \cite{3}, and $\kappa ^h(G)$ requires that any one vertex of every component left behind has at least $h$ neighborts.
In the Network-on-Chip technology, if there is a fault node on the chip, then we assume that the whole chip is fault. Because of the facts,
structure (substructure) connectivity is defined by Lin et al. \cite {5}.
Let $\mathcal{F}$ be a set of subgraphs of $G$. If $\mathcal{F}$ disconnects $G$ or $G-\mathcal{F}$ is a trivial graph, then $\mathcal{F}$
 is called a subgraph cut of $G$.
For a connected subgraph $H$ and a subgraph cut $\mathcal{F}$ of $G$, if $\mathcal{F}$'s every member is isomorphic to $H$ (resp. a connected subgraph of $H$), then $\mathcal{F}$ is called an $H$-structure cut (resp. $H$-substructure cut). The size of the smallest $H$-structure cuts (resp. $H$-substructure cuts) of $G$ is $H$-structure connectivity of $G$, $\kappa(G; H)$ (resp. $H$-substructure connectivity of $G$, $\kappa^{s}(G; H)$).
Thus, $\kappa^{s}(G; H)\leqslant \kappa(G; H)$. Especially, $K_{1}$-structure connectivity is equal to vertex connectivity. For the past few years, there are  many studies on structure connectivity for some well-known networks, such as  $HL$-networks \cite {4}, alternating group graphs \cite{7}, star graph \cite{8}, wheel network \cite {9} and so on.

Data center network, DCN, is a networking infrastructure inside a data center, which connects many servers by links and switches \cite {10}.  Due to the the advances of science and technology,
DCN runs thousands of servers.
For instance, Google has over 450,000 servers operating in 30 data centers by 2006 \cite{11, 12}, and Microsoft and Yahoo! have hundreds of millions of servers running at the same time in their data centers \cite{13, 14}.
Then how to efficiently connect numerous servers become a fundamental challenge in DCN. Therefore, DCell and BCDC were proposed by Guo et al. \cite{10} and Wang et al. \cite{15}, respectively.
DCell defines by a recursive structure. A server connects to distinct levels of DCells by multiple links (for details see Definition 1). According to its structure, with the increase of node degree, DCell shows a doubly exponential growth. Then, without the use of costly core-switches or core-routers, a DCell with a degree less than 4 can operate over thousands of servers \cite{10}.
An $n$-dimensional BCDC, $B_n$, consists of an independent set and two $(n-1)$-dimensional BCDC's (for details see Definition 4).
Then the quantity of servers in BCDC add up fast as the dimension of BCDC grows. Such as, with the use of 16-port switches, $B_{16}$ operates 524 288 servers at once \cite{15}.
It is easy to find that DCell and BCDC provide two ways to connect plentiful servers efficiently. For these two networks, the best application scenarios are large data centers,
which contain end-user applications (e.g. Web search and IM) and distributed system operations (e.g. MapReduce and GFS) \cite{10}. Because DCell and BCDC based on complete graphs and crossed cubes, respectively, they still have lots of good properties, such as high-capacity, minor diameter and high fault-tolerance.
There are a few primary characters of graphs $D_{m,n}$ and $B_n$, such as connectivity, diameter, symmetry, broadcasting, have been studied recently \cite{10, 15}. As extensions of connectivity, $h$-restricted connectivity \cite{16, 17} and $h$-extra connectivity \cite{18, 17} have been studied, yet. Then, we compute the structure connectivities of $D_{m,n}$ and $B_n$ on some common structures, star, complete graph, path and cycle in this paper.

The paper has the following sections. Section 2 defines a few notations. Sections 3 presents the star-structure (substructure) connectivity and complete graph-structure connectivitives of DCell. Sections 4 shows the star-structure (substructure) connectivity, path-structure (substructure) connectivity and cycle-structure (substructure) connectivity of BCDC.  Section 5 summarizes the content of this paper and  indicates some unsolved questions.
\section{Notations}
In this paper, we only consider finite and simple graphs. For a graph $H$ that satisfies the conditions $V (H)\subseteq V(G)$ and $E(H) \subseteq E(G)$, we call $H$ a subgraph of $G$. For a vertex subset $S$ of $G$,
$G-S$ is a graph obtained by deleting the vertex set $S$ and all edges incident to them from $G$. If $H$ is a subgraph of $G$, then we set $G-H=G-V(H)$.
 Let $\mathcal{F}$ be a set whose members are subgraphs of $G$. Then $G-\mathcal{F}=G-V(\mathcal{F})$, where
 $V(\mathcal{F})$ is the union of the vertex set of members of $\mathcal{F}$.
 For a vertex subset $V'$ of $G$, we set $G[V']$ as the induced subgraph by $V'$ of $G$, where $V(G[V']) = V'$ and $E(G[V']) = \{(u,u')\in E(G)| u,u'\in V'\}.$
For any one vertex $x$ of $G$, $N_{G}(x)=\{y|xy\in E(G)\}$. For a vertex subset $A$ of $G$, $N_{G}(A)=\cup_{x\in A}N_{G}(x)\setminus A$.
Let $P_k=\langle v_{1}, v_{2}, \ldots,  v_{k}\rangle$ be a path. We set $P^{-1}_k=\langle v_{k}, v_{k-1}, \ldots, v_{1}\rangle$ and $P_k-v_i\cup w=\langle v_{1}, \ldots, v_{i-1}, w, v_{i+1}, \ldots,  v_{k}\rangle$.
If $v_1=v_k$ in $P_k$, $k\geq3$, then we call it cycle $C_k$. Let $\{x; x_i|1\leq i\leq t\}$ be a star $K_{1,t}$. Then we set $x$ and  $x_i$, $1\leq i\leq t$, are center and leaves, respectively.

\section{The structure connectivity of DCell }
Given an integer $s$, we set $\langle s\rangle=\{0,1,\ldots, s\}$ and $[s]=\{1,2,\ldots,s\}$. Define $I_{0,n} = \langle n-1\rangle$ and $I_{i,n} = \langle t_{i-1,n}\rangle$ for $i\in [m]$.
An $m$-dimensional DCell with $n$-port switches is denoted by $D_{m,n}$ for $m\geq 0$ and $n\geq 2$. Let $t_{m,n}$ be the number of servers in $D_{m,n}$, where $t_{0,n}=n$ and $t_{m,n}=t_{m-1,n} \cdot (t_{m-1,n} + 1)$ for $m\geq 1$. Servers  of $D_{m,n}$ can be labeled by $\{x_mx_{m-1} \ldots x_1x_0|x_i \in I_{i,n}, i\in \langle m\rangle\}$.

\begin{Def} \cite{10} The $m$-dimensional DCell with $n$-port switches $D_{m,n}$ is defined recursively as follows.
\begin{enumerate}[(1)]
\item \label{cond 1}
$D_{0,n}$ is a complete graph consisting of $n$ servers.
\item \label{cond 2}
For $m\geq 1$, $D_{m,n}$ is obtained from $t_{m-1,n} + 1$ disjoint copies $D_{m-1,n}$ by the following steps.
\begin{enumerate}[(i)]
\item \label{cond 1} Let $D^i_{m-1,n}$ be a copy of $D_{m-1,n}$ by prefixing label of each server with $i$ for $i \in I_{m,n}.$
\item \label{cond 2} Server $u = u_mu_{m-1},\ldots, u_0$ in $D^{u_m}_{m-1,n}$ is adjacent to server $v = v_mv_{m-1},\ldots, v_0$ in $D^{v_m}_{m-1,n}$ if and only if $u_m = v_0 + \sum^{m-1}_{j=1}(v_j \times t_{j-1,n})$
and $v_m = u_0 +\sum^{m-1}_{j=1} (u_j \times t_{j-1,n})+ 1$ for any $u_m, v_m \in I_{m,n}$ and $u_m< v_m$.
\end{enumerate}
\end{enumerate}
\end{Def}
If $u$ has a neighbor $u'\notin D^i_{m-1, n}$ for $u\in D^i_{m-1, n}$ and $i\in I_{m,n}$, then $u'$ is called the outside neighbor of $u$.
Figure 1(a) shows 1-dimensional DCell with 4-port switches.
Since switches are transparent in networks, the graph structure of $D_{1,4}$ is shown in Figure 1(b).
\begin{figure}[!htbp]
\begin{center}
\subfigure[
]{\includegraphics[totalheight=6.5cm, width=6.5cm]{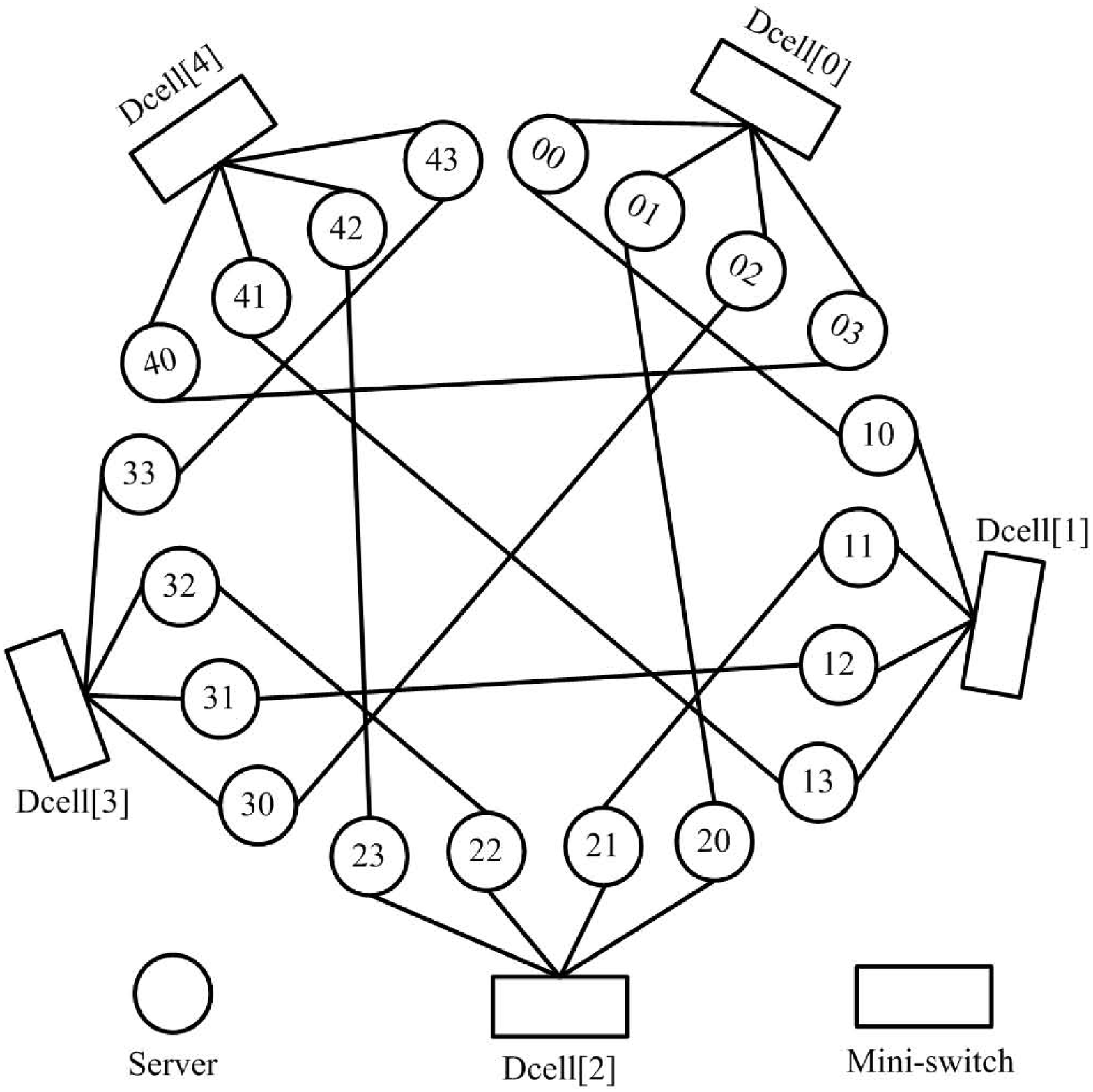}}
~~~~~\subfigure[ 
]{\includegraphics[totalheight=6cm, width=6cm]{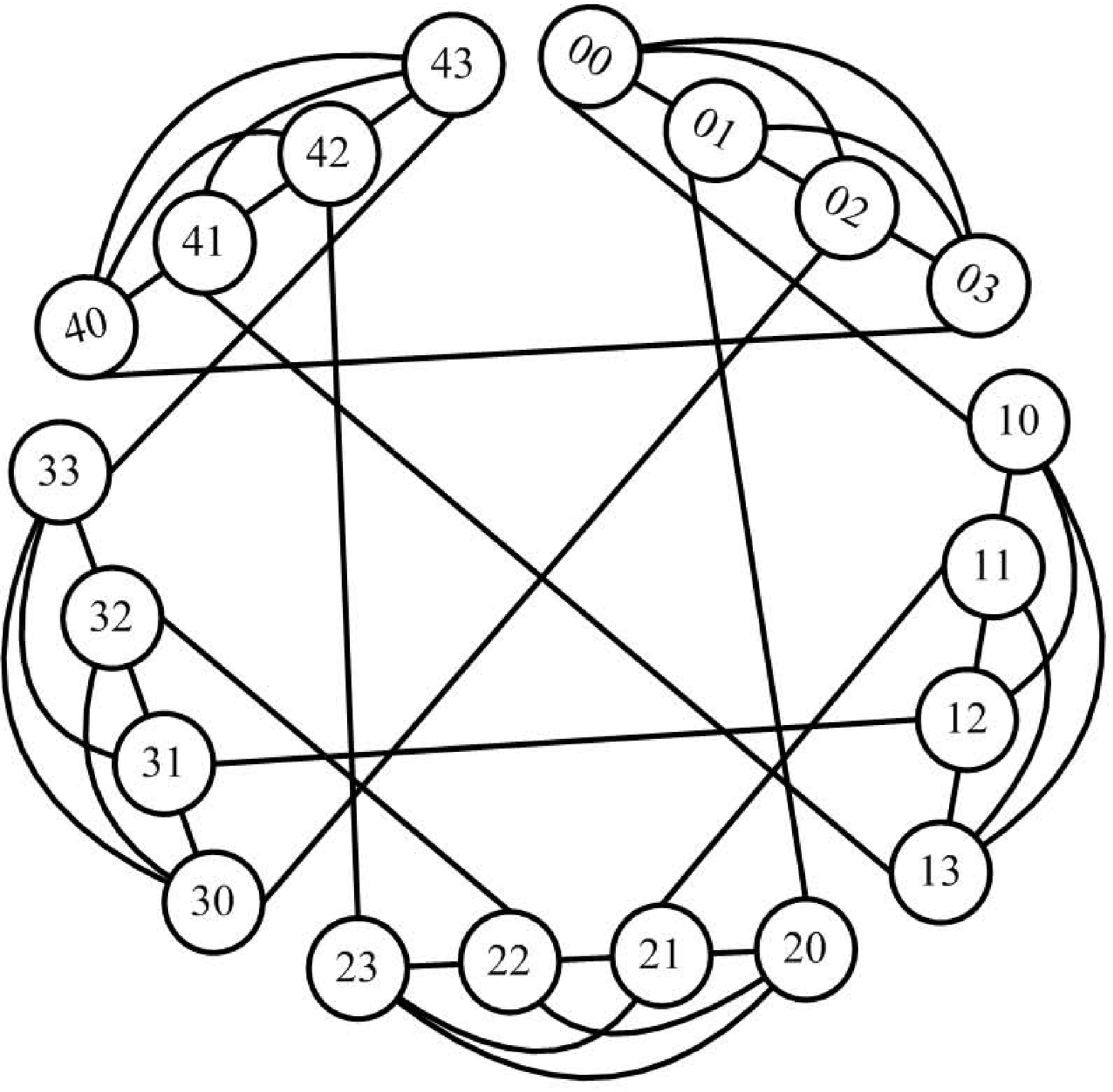}}
\caption{\label{ }\small{The 1-dimensional DCell with 4-port switches and the graph structure of $D_{1,4}$.}}
\end{center}
\end{figure}

\begin{Lem} \emph{\cite{19}} For $m\geq 0$ and $n \geq 2$, the following statements hold:
\begin{enumerate}[(i)]
\item \label{cond 1}
$D_{m,n}$ is $(m+n-1)$-regular and $|V_{m,n}| = t_{m,n}$,
\item \label{cond 2}
for $m\geq 1$, $D_{m,n}$ consists of  $t_{m-1,n}+1$ copies of $D_{m-1,n}$, denoted by $D^i_{m-1,n}$, for each $i\in I_{m,n}$. There is only one edge between $D^i_{m-1,n}$ and $D^j_{m-1,n}$ for any $i,j \in I_{m,n}$ and $i\neq j$, which implies that the outside neighbors of vertices in $D^i_{m-1,n}$ belong to different copies of $D^j_{m-1,n}$ for $i,j \in I_{m,n}$ and $i\neq j$.
\end{enumerate}
\end{Lem}

For convenience, let
\begin{equation}\label{Eq}
D_{m,n}=D^0_{m-1,n}\otimes D^1_{m-1,n}\otimes D^2_{m-1,n}\otimes \ldots \otimes D^{t_{m-1,n}}_{m-1,n}.
\end{equation}

\begin{Lem} \emph{\cite{10}} For $m\geq 0$ and $n\geq 2$, $\kappa (D_{m,n})=n+m-1$ and 
 $t_{m,n}\geq (n+\frac{1}{2})^{2^m}-\frac{1}{2}.$
\end{Lem}

\begin{Lem} $\kappa^s (D_{m,n}; K_{1,t})\geq \lceil \frac{n-1}{1+t}\rceil+m$ for $1\leq t\leq m+n-1.$
\end{Lem}
\begin{proof} Let's explain this by mathematical induction on $m$. For $m=0$, $D_{0,n}\cong K_n$. Assume $\kappa^s (D_{0,n}; K_{1,t})\leq \lceil \frac{n-1}{1+t}\rceil-1$ for $1\leq t\leq n-1.$ Then $D_{0,n}$ has a $K_{1,t}$-substructure cut $\mathcal{F'}$ with $|\mathcal{F'}|\leq \lceil\frac{n-1}{1+t}\rceil-1$.
We have
\begin{equation*}
|V(\mathcal{F'})|\leq (1+t)(\lceil\frac{n-1}{1+t}\rceil-1)\leq (1+t)(\frac{n+t-1}{1+t}-1)=n-2.
\end{equation*}
Since $\kappa (D_{0,n})=n-1$, $|V(\mathcal{F'})|< \kappa (D_{0,n})$, a contradiction. Then it is true for $m=0$.

Suppose $\kappa^s (D_{m-1,n}; K_{1,t})\geq \lceil \frac{n-1}{1+t}\rceil+m-1$ for $1\leq t\leq m+n-2.$
 Now, we consider $D_{m,n}$. Let $\mathcal{F}=\{T_i|0\leq i\leq \lceil\frac{n-1}{1+t}\rceil+m-2\}$ be a set of connected subgraphs of $K_{1,t}$ for $1\leq t\leq n+m-1$. We need to show that $D_{m,n}-\mathcal{F}$ is connected. Let $\mathcal{F}^i=\mathcal{F}\cap D^i_{m-1, n}$,
$\mathcal{F}^{M}=\cup_{i\in M}\mathcal{F}^i$ and $D_{m-1,n}^{M}=\cup_{i\in M} D^i_{m-1,n}$ by Eq. (\ref{Eq}). We know that each star is in at most two $D^i_{m-1,n}$'s by Definition 1. Then there exist at most
$2(\lceil \frac{n-1}{1+t}\rceil+m-1)$ $D^i_{m-1,n}$'s such that $\mathcal{F}^i\neq \emptyset.$ Let $S=\{0,1,2,\ldots, s\}$ be a set such that $\mathcal{F}^i\neq\emptyset$ for $i\in S$. Then $|S| \leq 2(\lceil \frac{n-1}{1+t}\rceil+m-1)$.
By Lemma 3.2, for $m\geq 1$ and $n\geq 2$,
$$
\aligned
&t_{m-1,n}+1-|S|\geq t_{m-1,n}+1-2(\lceil \frac{n-1}{1+t}\rceil+m-1)
\geq t_{m-1,n}+1-2( \frac{n+t-1}{1+t}+m-1)\\&
\geq t_{m-1,n}+1-2( \frac{n}{2}+m-1)
\geq (n+\frac{1}{2})^{2^{m-1}}-n-2m+\frac{5}{2}
\geq (n+\frac{1}{2})-n-2+\frac{5}{2}
=1,
\endaligned
$$
which implies that there exists $D^i_{m-1,n}$ satisfied $\mathcal{F}^i=\emptyset$ for $i\in I_{m,n}\setminus S$.
Then $\mathcal{F}^{I_{m,n}\setminus S}=\emptyset$ and $D_{m-1,n}^{I_{m,n}\setminus S}-\mathcal{F}^{I_{m,n}\setminus S}=D_{m-1,n}^{I_{m,n}\setminus S}$ is connected. Let $S_1=\{0,1,2,\ldots, s_1\}$ be a subset of $S$ satisfied that there exist star-centers in $\mathcal{F}^i$ for $i\in S_1$. Then there only leaves of stars in $\mathcal{F}^i$ for $i\in S\setminus S_1$.

If there exist $D^i_{m-1,n}$ and $D^j_{m-1,n}$ such that both $D^i_{m-1,n}-\mathcal{F}^i$ and $D^j_{m-1,n}-\mathcal{F}^j$ are disconnected for $i\in S$, $j \in S$ and $i\neq j$, then $|\mathcal{F}^i|\geq \lceil \frac{n-1}{1+t}\rceil+m-1$ and $|\mathcal{F}^j|\geq \lceil \frac{n-1}{1+t}\rceil+m-1$ by induction hypothesis. Since $|\mathcal{F}|=\lceil \frac{n-1}{1+t}\rceil+m-1$, $|\mathcal{F}^i|\geq \lceil \frac{n-1}{1+t}\rceil+m-1$ and $|\mathcal{F}^j|\geq \lceil \frac{n-1}{1+t}\rceil+m-1$ are impossible by Definition 1. Thus there is at most one $D^i_{m-1,n}$ satisfied that $D^i_{m-1,n}-\mathcal{F}^i$ is disconnected for $i\in S$.

  \textbf{Case 1}. Each $D^i_{m-1,n}-\mathcal{F}^i$ is connected for $i\in S$.

  For each $i\in S\setminus S_1$, $\mathcal{F}^i$ consists of the outside neighbors of partical star-centers which are in $D^{S_1}_{m-1,n}$. Then $D^i_{m-1,n}-\mathcal{F}^i$ has an outside neighbor of a vertex which is in $D^{I_{m,n}\setminus S}_{m-1, n}$ for $i\in S\setminus S_1$. Since $D^i_{m-1,n}-\mathcal{F}^i$ is connected,
  $D_{m-1,n}^{I_{m,n}\setminus S_1}-\mathcal{F}^{I_{m,n}\setminus S_1}$ is connected.
   Let $p_i$ be the number of star-centers in $D^i_{m-1,n}$ for $i\in S_1$. For $p_i=\lceil \frac{n-1}{1+t}\rceil+m-1$, we find $S_1=\{i\}$. Then each vertex of $D^i_{m-1,n}-\mathcal{F}^i$ has an outside neighbor in $D_{m-1,n}^{I_{m,n}\setminus S}$. Since $D^{I_{m,n}\setminus \{i\}}_{m-1,n}-\mathcal{F}^{I_{m,n}\setminus \{i\}}$ is connected, $D_{m,n}-\mathcal{F}$ is connected (see Figure \ref{L3}(a)). For $p_i\leq\lceil \frac{n-1}{1+t}\rceil+m-2$, by Lemma 3.2, we have
 for $n\geq 2$ and $m\geq 2$,
 $$
 \aligned
 &|V(D^i_{m-1,n})|-(1+t)p_i-(\lceil \frac{n-1}{1+t}\rceil+m-1-p_i)
 =t_{m-1,n}-tp_i-(\lceil \frac{n-1}{1+t}\rceil+m-1)
 \geq\\& t_{m-1,n}-t(\lceil \frac{n-1}{1+t}\rceil+m-2)-(\lceil \frac{n-1}{1+t}\rceil+m-1)
 =t_{m-1,n}-(1+t)(\lceil \frac{n-1}{1+t}\rceil+m-1)+t\\
 &\geq t_{m-1,n}-(1+t)( \frac{n+t-1}{1+t}+m-1)+t
 \geq (n+\frac{1}{2})^{2^{m-1}}-n-m-(m-1)t+\frac{3}{2}
 \\&\geq(n+\frac{1}{2})^{2^{m-1}}-n-m-(m-1)(n+m-1)+\frac{3}{2}
 =(n+\frac{1}{2})^{2^{m-1}}-mn-m^2+m+\frac{1}{2}\\
 &\geq (2+\frac{1}{2})^{2^{m-1}}-2m-m^2+m+\frac{1}{2}
 \geq(2+\frac{1}{2})^{2}-4-4+2+\frac{1}{2}>0,
 \endaligned
 $$
and for $m=1$,
  $$
 \aligned
 |V(D^i_{0,n})|-(1+t)p_i-(\lceil \frac{n-1}{1+t}\rceil-p_i)
 \geq t_{0,n}-(1+t)( \frac{n+t-1}{1+t}+1-1)+t>0,
 \endaligned
 $$
which implies that $D^i_{m-1,n}-\mathcal{F}^i$ has a vertex that is connected to $D_{m-1,n}^{I_{m,n}\setminus S_1}-\mathcal {F}^{I_{m,n}\setminus S_1}$ for each $i\in S_1$ (see Figure \ref{L3}(b)). Thus $D_{m,n}-\mathcal{F}$ is connected.

\textbf{Case 2}. There is exactly one $D^i_{m-1,n}$ such that $D^i_{m-1,n}-\mathcal{F}^i$ is disconnected for $i\in S$.

Since $D^j_{m-1,n}-\mathcal{F}^j$ is connected for $j\in S$ and $j\neq i$, $D^{I_{m,n}\setminus \{i\}}_{m-1,n}-\mathcal{F}^{I_{m,n}\setminus \{i\}}$ is connected by the similar argument as \textbf{Case 1}.
By induction hypothesis, $|\mathcal{F}^i|\geq \lceil \frac{n-1}{1+t}\rceil+m-1$.
If $i\in S_1$, then each $\mathcal{F}^j$ has exactly one star-center and exactly one of leaves of that star in $\mathcal{F}^i$ for $j\in S_1\setminus \{i\}$, and each $\mathcal{F}^j$ has exactly one leaf of star and its center in $\mathcal{F}^i$ for $j\in S\setminus S_1$ (see Figure \ref{L3}(c)). By Lemma 3.1$(ii)$, each vertex of $D^i_{m-1,n}-\mathcal{F}^i$ has an outside neighbor in $D^{I_{m,n}\setminus S}_{m-1,n}$. Thus $D_{m,n}-\mathcal{F}$ is connected.
If $i\in S\setminus S_1$, then $S\setminus S_1=\{i\}$ and $|S_1|=\lceil\frac{n-1}{1+t}\rceil+m-1$. Otherwise, $|\mathcal{F}^i|< \lceil \frac{n-1}{1+t}\rceil+m-1$, a contradiction. For $j\in S_1$, we have that each $\mathcal{F}^j$ has exactly one star-center and one leaf of the star is in $\mathcal{F}^i$ (see Figure \ref{L3}(d)). Thus $\mathcal{F}^i$ has all outside neighbors of vertices which from different $D^j_{m-1,n}$s, $j\in S\setminus \{i\}$. By Lemma 3.1$(ii)$, each vertex of $D^i_{m-1,n}-\mathcal{F}^i$ has an outside neighbor in $D^{I_{m,n}\setminus S}_{m-1,n}$. Thus $D_{m,n}-\mathcal{F}$ is connected.
\end{proof}
\begin{figure}[!htbp]
\centering
\subfigure[]{\includegraphics[totalheight=4.3cm, width=7.5cm]{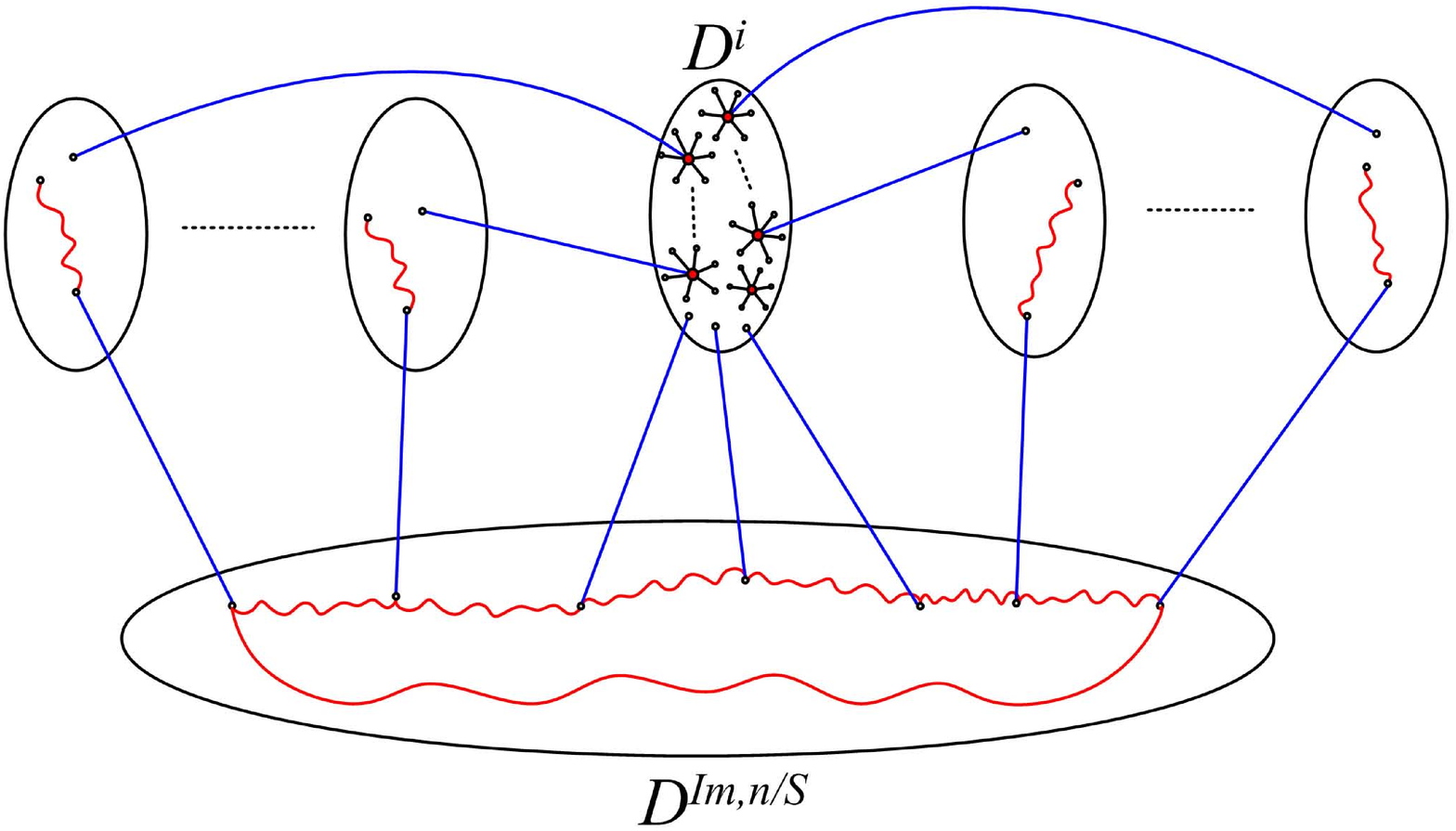}}
~~\subfigure[]{\includegraphics[totalheight=4.5cm, width=8cm]{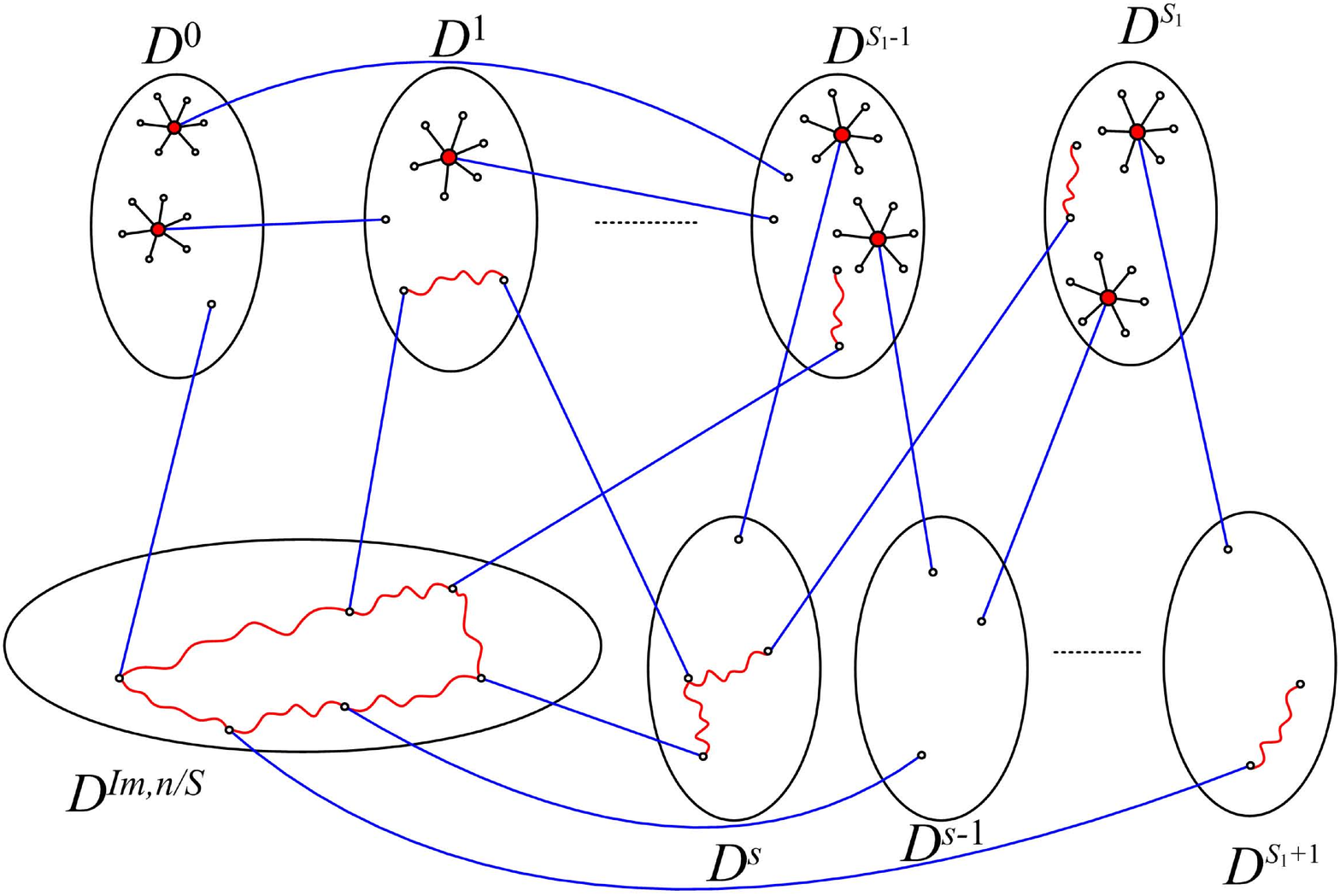}}\\
\subfigure[]{\includegraphics[totalheight=4.5cm, width=8.5cm]{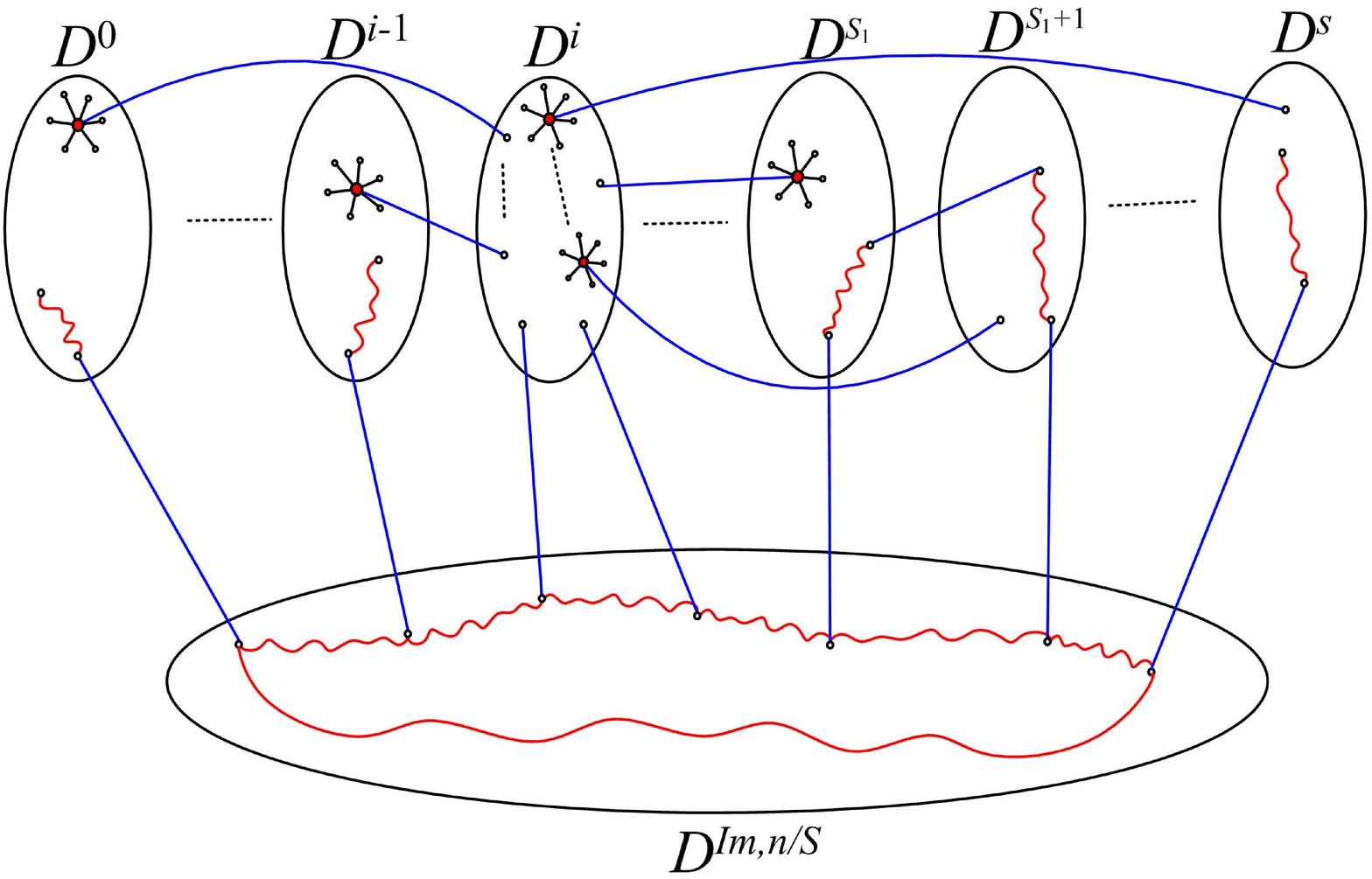}}
~~\subfigure[]{\includegraphics[totalheight=4.5cm, width=7cm]{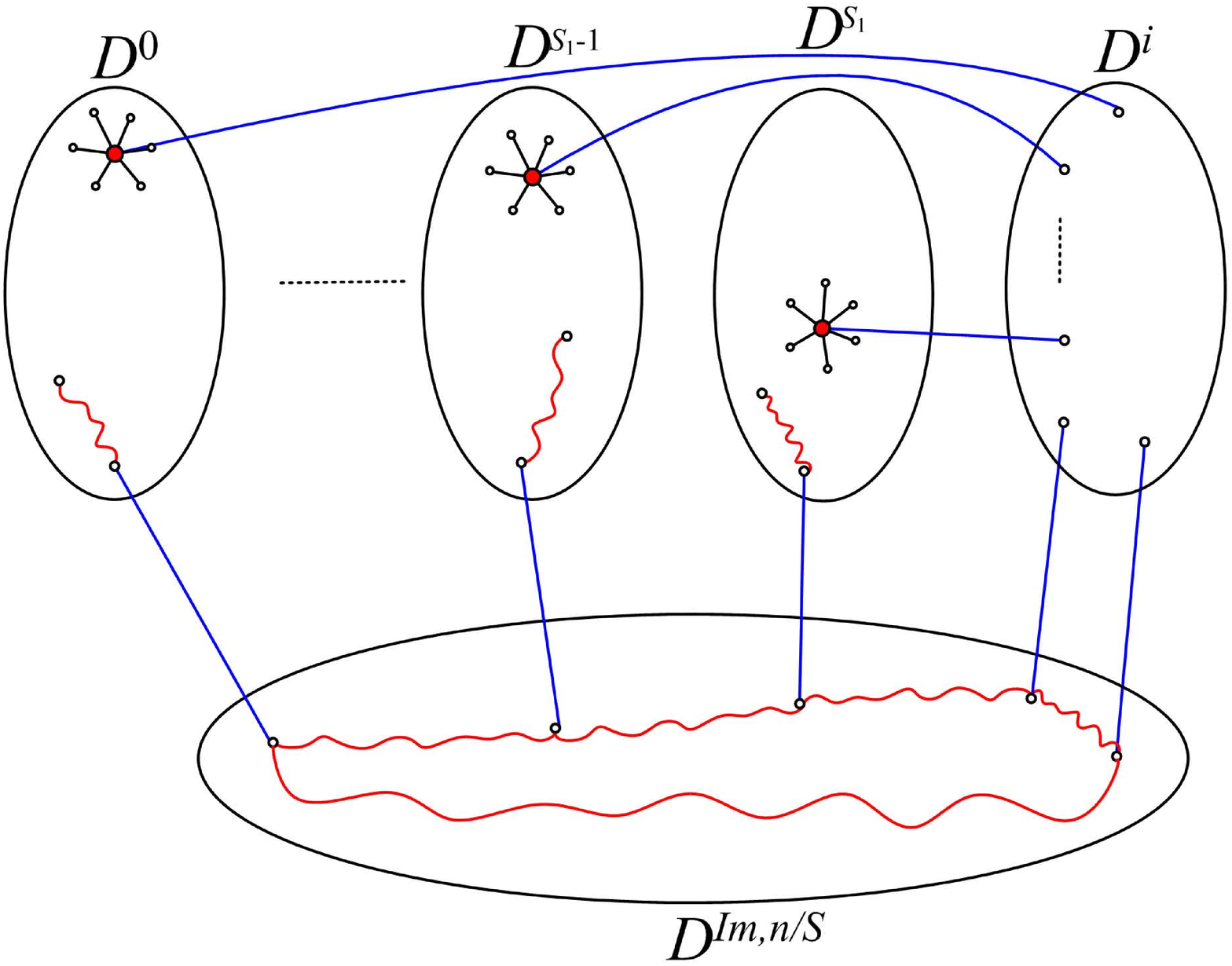}}\\
\caption{\label{L3}Illustration for Lemma 3.3.}
\end{figure}

\begin{Lem}$\kappa (D_{m,n}; K_{1,t})\leq \lceil \frac{n-1}{1+t}\rceil+m$ for $1\leq t\leq m+n-2.$
\end{Lem}
\begin{proof} Given two vertices $u=00\ldots 00$ and $v=11\ldots 11$ in $D_{m,n}$. Then $N_{D_{m,n}}(u)=\{00\ldots 0i|i\in [n-1]\}\cup \{10\ldots 00, 01\ldots 00, \ldots, 00\ldots 10 \}$. Since $D_{m,n}[\{00\ldots 0i|i\in [n-1]\}]$ is a complete graph $K_{n-1}$ by Definition 1, we need $\lceil\frac{n-1}{1+t}\rceil$ $K_{1,t}$'s to cover all vertices of $D_{m,n}[\{00\ldots 0i|i\in [n-1]\}]$. Then we set
\begin{equation*}
S_i=\left\{
\begin{aligned}
&\{00\ldots 0[(i-1)t+i]; 00\ldots 0[(i-1)t+i+k]|1\leq k\leq t\},~~for~~ 1\leq i\leq \lfloor\frac{n-1}{1+t}\rfloor;\\
&\{00\ldots 0(n-1); 00\ldots 0k|n-t-1\leq k\leq n-2\},~~for~~i=\lceil\frac{n-1}{1+t}\rceil.\\
\end{aligned}
\right.
\end{equation*}
Next, we will construct several stars $K_{1,t}$'s to cover the vertex subset $\{10\ldots 00, 01\ldots 00, \ldots, 00\ldots 10 \}$. For $1\leq j\leq m$, let $S_{\lceil\frac{n-1}{1+t}\rceil+j}$ be a star with center $0\ldots 010\ldots 0$ ($u_j=1$) and its leaves are the neighbors of $0\ldots 010\ldots 0$ except $u$ (see Figure 3(a)).

Let $\mathcal{F}=\{S_i|i\in [\lceil \frac{n-1}{1+t}\rceil+m]\}$.
Since $N_{D_{m,n}}(u)\subseteq V(\mathcal{F})$ and $v\in V(D_{m,n}-\mathcal{F})$, $D_{m,n}-\mathcal{F}$ has two components. Then $u$ is a singleton and $v$ belongs to the other component. Thus $\kappa (D_{m,n}; K_{1,t})\leq |\mathcal{F}|=\lceil \frac{n-1}{1+t}\rceil+m$.

\end{proof}

Lemmas 3.3 and 3.4 imply that  $\lceil \frac{n-1}{1+t}\rceil+m\leq \kappa^s (D_{m,n}; K_{1,t})\leq \kappa (D_{m,n}; K_{1,t})\leq \lceil \frac{n-1}{1+t}\rceil+m$ for $1\leq t\leq m+n-2$. Then following result holds.
\begin{The}$\kappa (D_{m,n}; K_{1,t})=\kappa^s(D_{m,n}; K_{1,t})=\lceil \frac{n-1}{1+t}\rceil+m$ for $1\leq t\leq m+n-2$.
\end{The}

\begin{Lem} $\kappa (D_{m,n}; K_s)\geq \lceil\frac{n-1}{s}\rceil+m$ for $3\leq s\leq n-1$.
\end{Lem}
\begin{proof} Induction on $m$. For $m=0$, $D_{0,n}\cong K_n$. Assume $\kappa (D_{0,n}; K_s)\leq \lceil\frac{n-1}{s}\rceil-1$ for $3\leq s\leq n-1$. Then there exists a $K_s$-structure cut $\mathcal{F'}$ of $D_{0,n}$, $|\mathcal{F'}|\leq \lceil\frac{n-1}{s}\rceil-1.$ We have
\begin{equation*}
|V(\mathcal{F'})|\leq s\cdot (\lceil\frac{n-1}{s}\rceil-1)\leq s\cdot (\frac{n+s-2}{s}-1)=n-2<n-1.
\end{equation*}
Since $\kappa (D_{0,n})=n-1$, $|V(\mathcal{F'})|< \kappa (D_{0,n})$, a contradiction. Then $\kappa (D_{0,n}; K_s)\geq \lceil\frac{n-1}{s}\rceil$ for $3\leq s\leq n-1$.

 Assume $\kappa(D_{m-1,n}; K_s)\geq  \lceil\frac{n-1}{s}\rceil+m-1$ for $3\leq s\leq n-1$.
 Now, we consider $D_{m,n}$. Let $\mathcal{F}=\{T_i|1\leq i\leq \lceil\frac{n-1}{s}\rceil+m-1\}$ be a set such that $T_i\cong K_s$ for $1\leq i\leq \lceil\frac{n-1}{s}\rceil+m-1$. We need to show that $D_{m,n}-\mathcal{F}$ is connected.
Let $\mathcal{F}^i=\mathcal{F}\cap D^i_{m-1, n}$,
$\mathcal{F}^{T}=\cup_{i\in T}\mathcal{F}^i$ and $D_{m-1,n}^{T}=\cup_{i\in T} D^i_{m-1,n}$ by Eq. (\ref{Eq}).
We know that a complete graph $K_s$ for each $3\leq s\leq n-1$ is in exactly one $D^i_{m-1,n}$ by Definition 1. Let $p_i$ be the number of $K_s$'s in $D^i_{m-1,n}$.
If $p_i=\lceil\frac{n-1}{s}\rceil+m-1$ for some $i\in I_{m,n}$, then $\mathcal{F}^{I_{m,n}\setminus \{i\}}=\emptyset$. Thus $D^{I_{m,n}\setminus \{i\}}_{m-1,n}-\mathcal{F}^{I_{m,n}\setminus \{i\}}$ is connected. By Definition 1, each vertex of $D^i_{m-1,n}-\mathcal{F}^i$ has an outside neighbor in
$D^{I_{m,n}\setminus \{i\}}_{m-1,n}$. Then $D_{m,n}-\mathcal{F}$ is connected.

If $p_i\leq \lceil\frac{n-1}{s}\rceil+m-2$ for some $i\in I_{m,n}$, then $D^i_{m-1,n}-\mathcal {F}^i$ is connected by induction hypothesis. Since $|\mathcal{F}|=\lceil\frac{n-1}{s}\rceil+m-1$,
there are at most $(\lceil\frac{n-1}{s}\rceil+m-1)$ $D^i_{m-1,n}$'s such that $\mathcal{F}^i\neq \emptyset$. Let $X=\{0,1,\ldots, x\}$ be a set such that $\mathcal{F}^i\neq\emptyset$ for $i\in X$. Then $|X| \leq \lceil \frac{n-1}{s}\rceil+m-1$.
By Lemma 3.2,
$$
\aligned
&t_{m-1,n}+1-|X|
\geq t_{m-1,n}+1-(\lceil \frac{n-1}{s}\rceil+m-1)
\geq t_{m-1,n}+1-(\frac{n+s-2}{s}+m-1)\\
&\geq (n+\frac{1}{2})^{2^{m-1}}-\frac{1}{2}+1-(\frac{n+1}{3}+m-1)
\geq n+1-\frac{n+1}{3}
\geq 2,
\endaligned
$$
which implies that there exist $D^i_{m-1,n}$'s such that $\mathcal{F}^i=\emptyset$ for $i\in I_{m,n}\setminus X$.
Then $\mathcal{F}^{I_{m,n}\setminus X}=\emptyset$ and $D_{m-1,n}^{I_{m,n}\setminus X}-\mathcal{F}^{I_{m,n}\setminus X}$ is connected.
By Lemma 3.2, we find, for $i\in X$, $m\geq 1$ and $n\geq 2$,
$$
 \aligned
 &|V(D^i_{m-1,n})|-sp_i-(\lceil \frac{n-1}{s}\rceil+m-1-p_i)
 =t_{m-1,n}-(s-1)p_i-(\lceil \frac{n-1}{s}\rceil+m-1)\\&\geq
t_{m-1,n}-(s-1)(\lceil \frac{n-1}{s}\rceil+m-2)-(\lceil \frac{n-1}{s}\rceil+m-1)
=t_{m-1,n}-s(\lceil \frac{n-1}{s}\rceil+m-2)-1
 \\&\geq t_{m-1,n}-s(\frac{n+s-2}{s}+m-2)-1
 \geq (n+\frac{1}{2})^{2^{m-1}}-\frac{1}{2}-n-(m-1)s+1
 \\&\geq(n+\frac{1}{2})^{2^{m-1}}-n-(m-1)(n-1)+\frac{1}{2}
 =(n+\frac{1}{2})-n+\frac{1}{2}=1>0,
 \endaligned
 $$
which implies that $D^i_{m-1,n}-\mathcal{F}^i$ has a vertex that is connected to $D_{m-1,n}^{I_{m,n}\setminus X}$ for each $i\in X$. Since $D^i_{m-1,n}-\mathcal{F}^i$ is connected for $i\in X$, $D_{m,n}-\mathcal{F}$ is connected.
\end{proof}

\begin{Lem}$\kappa (D_{m,n}; K_s)\leq \lceil\frac{n-1}{s}\rceil+m$ for $3\leq s\leq n-1$.
\end{Lem}
\begin{proof} Given two vertices $u=00\ldots 00$ and $v=11\ldots 11$ in $D_{m,n}$. Then $N_{D_{m,n}}(u)=\{00\ldots 0i|i\in [n-1]\}\cup \{10\ldots 00, 01\ldots 00, \ldots, 00\ldots 10 \}$. Let $S_i=\{x_i|1\leq i\leq s\}$ be a complete graph $K_s$. Since $D_{m,n}[\{00\ldots 0i|i\in [n-1]\}]$ is a complete graph $K_{n-1}$ by Definition 1, we need $\lceil\frac{n-1}{s}\rceil$ $K_s$'s to cover all vertices of $D_{m,n}[\{00\ldots 0i|i\in [n-1]\}]$. Then we set
\begin{equation*}
S_i=\left\{
\begin{aligned}
&\{00\ldots 0[(i-1)s+k+1]|0\leq k\leq s-1\},~~for~~ 1\leq i\leq \lfloor\frac{n-1}{s}\rfloor;\\
&\{00\ldots 0k|n-s\leq k\leq n-1\},~~for~~i=\lceil\frac{n-1}{s}\rceil.\\
\end{aligned}
\right.
\end{equation*}

Next, we will construct several complete graphs $K_{s}$'s to cover the vertex subset $\{10\ldots 00, 01\ldots 00, \ldots, 00\ldots 10 \}$. For each $j\in [m]$,
let $S_{\lceil\frac{n-1}{s}\rceil+j}$ be a $K_s$ with vertices $0\ldots 010\ldots 0k$ ($u_j=1$ and $u_0=k$) for $0\leq k\leq s-1$.
Precisely, $S_{\lceil\frac{n-1}{s}\rceil+1}=\{00\ldots 01k|0\leq k\leq s-1\}$,
$S_{\lceil\frac{n-1}{s}\rceil+2}=\{00\ldots 10k|0\leq k\leq s-1\}$, analogously, $S_{\lceil\frac{n-1}{s}\rceil+m}=\{10\ldots 00k|0\leq k\leq s-1\}$ (see Figure 3(b)).

Let $\mathcal{F}=\{S_i|i\in [\lceil \frac{n-1}{s}\rceil+m]\}$.
Since $N_{D_{m,n}}(u)\subseteq V(\mathcal{F})$ and $v\in V(D_{m,n}-\mathcal{F})$, $D_{m,n}-\mathcal{F}$ has two components. Then $u$ is a singleton and $v$ belongs to the other component. Thus $\kappa (D_{m,n}; K_s)\leq |\mathcal{F}|=\lceil \frac{n-1}{s}\rceil+m$.
\end{proof}
\begin{figure}[!htbp]\label{}
\centering
\subfigure[A $K_{1,4}$-structure cut of $D_{6,8}$.]{\includegraphics[totalheight=8cm, width=8cm]{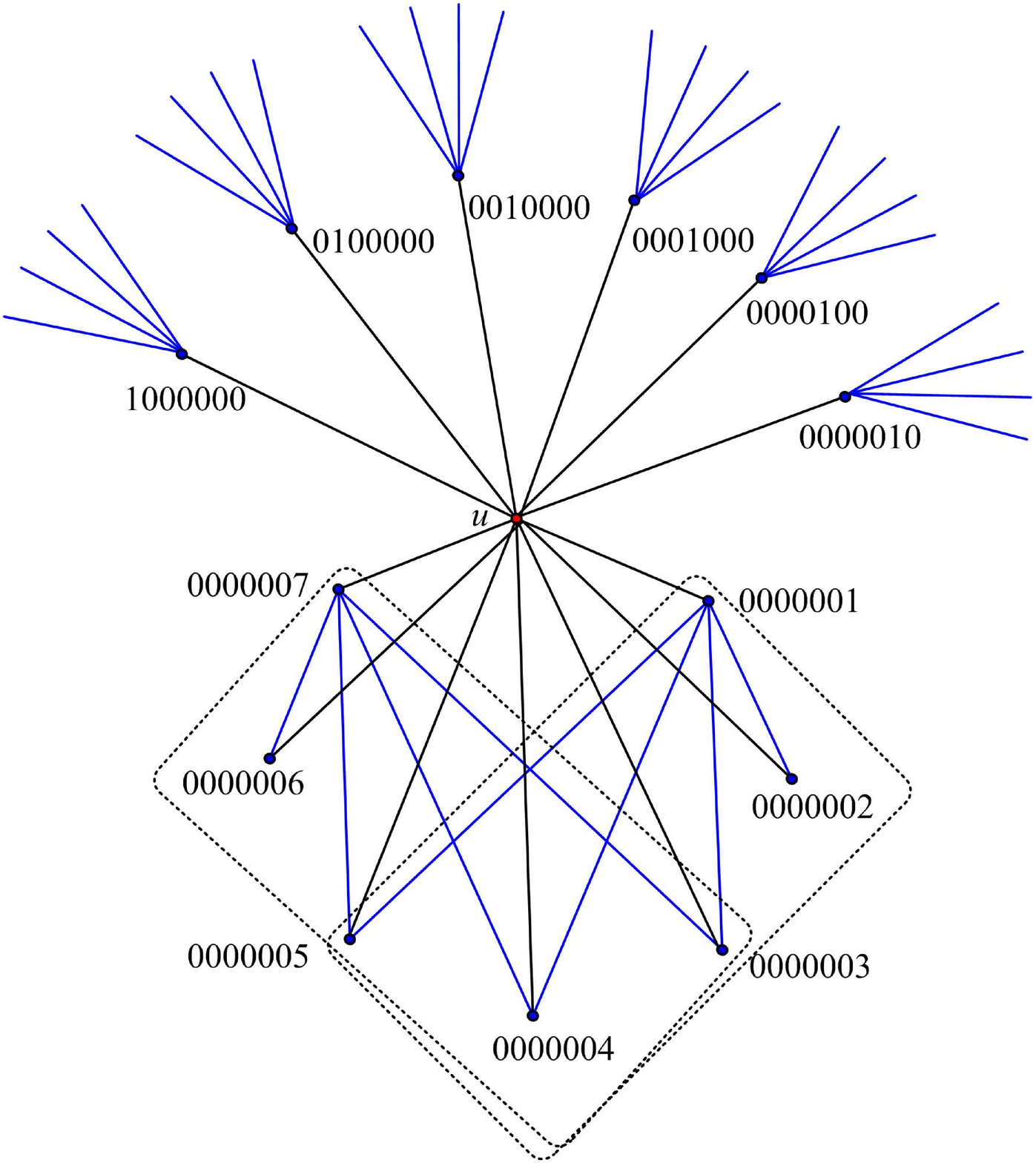}}
~~\subfigure[A $K_{5}$-structure cut of $D_{6,8}$.]{\includegraphics[totalheight=8cm, width=7.6cm]{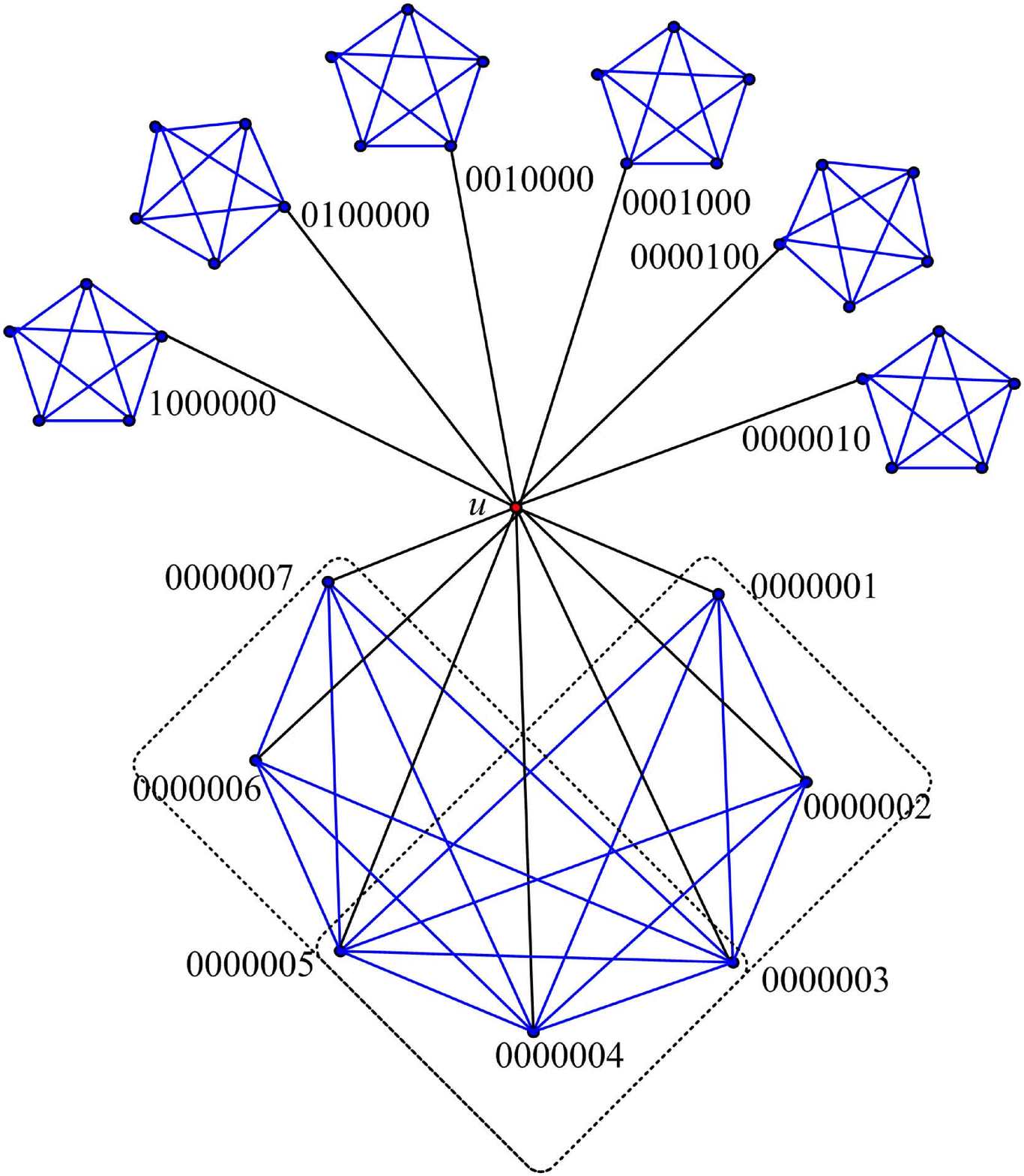}}\\
\caption{Illustration for Lemmas 3.4 and 3.7. }
\end{figure}
On the basis of 3.6 and 3.7, the following result holds.
\begin{The}$\kappa (D_{m,n}; K_s)=\lceil\frac{n-1}{s}\rceil+m$ for $3\leq s\leq n-1$.
\end{The}

\section{The structure connectivity of BCDC }
Since $n$-dimensional BCDC consists of two $n-1$-dimensional crossed cubes and an independent set, we introduce $n$-dimensional crossed cube first.
\begin{Def} \cite{20}
Two binary strings $x = x_1x_0$ and $y = y_1y_0$ are called pair related (denoted by $x\sim y$) if and only if $(x, y)\in \{(00, 00),(10, 10),(01, 11),(11, 01)\}$.
\end{Def}
\begin{Def} \cite{20}
We set $CQ_1$ as $K_2$ with vertices $0$ and $1$. $CQ_n$ consists of $CQ^0_{n-1}$ and
$CQ^1_{n-1}$. Two vertices $u = 0u_{n-2}\ldots u_0\in V (CQ^0_{n-1})$
and $v = 1v_{n-2}\ldots v_0 \in V(CQ^1_{n-1})$ are joined by an edge in $CQ_n$ if and only if
\begin{enumerate}[(i)]
\item \label{cond 1} $u_{n-2} = v_{n-2}$ if $n$ is even, and
\item \label{cond 2} $u_{2i+1}u_{2i} \sim v_{2i+1}v_{2i}$, for
$ 0\leq i <\lfloor \frac{n-1}{2} \rfloor$.
\end{enumerate}
\end{Def}
For any two vertices $u = u_{n-1}u_{n-2}\ldots u_0$ and $v = v_{n-1}v_{n-2}\ldots v_0$, $u$ is adjacent to $v$ if and only if there is $0\leq d\leq n-1$ satisfying the following conditions: (1) $u_{n-1}\ldots u_{d+1}=v_{n-1}\ldots v_{d+1}$,
(2) $u_d\neq v_d$, (3) $u_{d-1}=v_{d-1}$ if $d$ is odd, (4) $u_{2i+1}u_{2i} \sim v_{2i+1}v_{2i}$, for all
$ 0\leq i \leq\lfloor \frac{d}{2} \rfloor-1$. Then we set $v$  as $d$-dimensional neighbor of $u$, denoted by $u^d$.

BCDC's switches are vertices of $CQ_n$ and servers are edges of $CQ_n$. Each switch is an $n$-bit binary string $x = x_{n-1}x_{n-2}\ldots x_0$ and each server is a pair $[x,y]$. Then we get the original graph $A_n$ of the $n$-dimensional BCDC network, $B_n$.
Consider that switches are transparent, the graph structure of BCDC network as follows:

\begin{Def}\label{Bn} \cite{15}
We set $B_2$ as a $4$-cycle with vertices $[00, 01], [00, 10], [01, 11]$ and $[10, 11].$
For $n \geq 3$, $B^0_{n-1}$ (resp. $B^1_{n-1}$) is
obtained by prefixing each vertex $[x, y]$ of $B_{n-1}$ with $0$ (resp. $1$).
$B_n$ consists of $B^0_{n-1}$, $B^1_{n-1}$ and a vertex subset $S_n = \{[a, b]|a\in V (CQ^0_{n-1}), b\in V (CQ^1_{n-1})$ and $(a, b)\in E(CQn)\}$.
For three vertices $u = [a, b]\in V (B^0_{n-1}), v = [c, d]\in S_n$ and $w = [e, f]\in V (B^1_{n-1})$,
\begin{enumerate}[(i)]
\item \label{cond 1} $(u, v)\in E(B_n)$ if and only if $a = c$ or $b = c$;
\item \label{cond 2} $(v, w)\in  E(B_n)$ if and only if $e = d$ or $f = d$.
\end{enumerate}
\end{Def}
\begin{figure}[!htbp]
\centering
\subfigure[$A_{3}$.]{\includegraphics[totalheight=4cm, width=5cm]{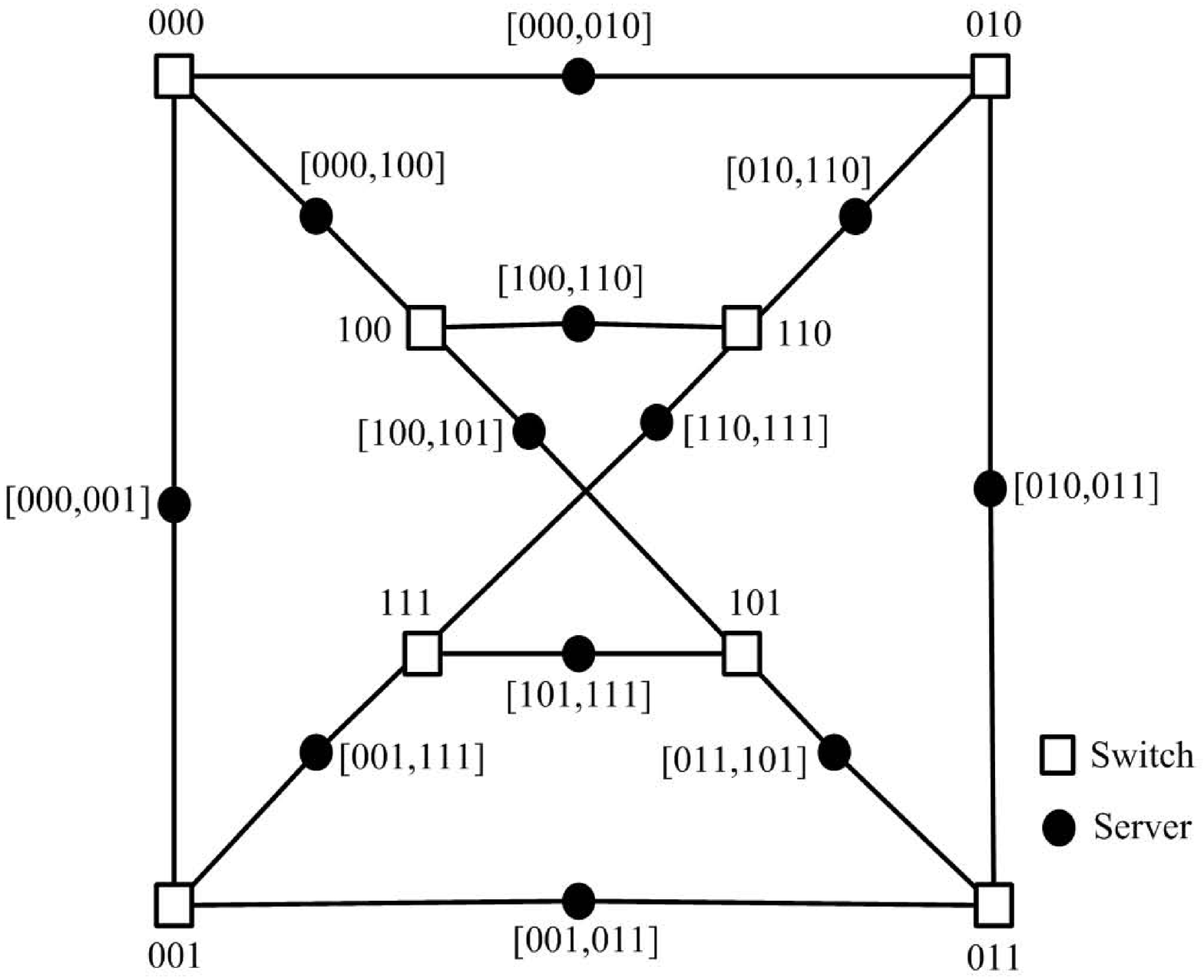}}
~~\subfigure[$CQ_{3}$.]{\includegraphics[totalheight=4cm, width=3.5cm]{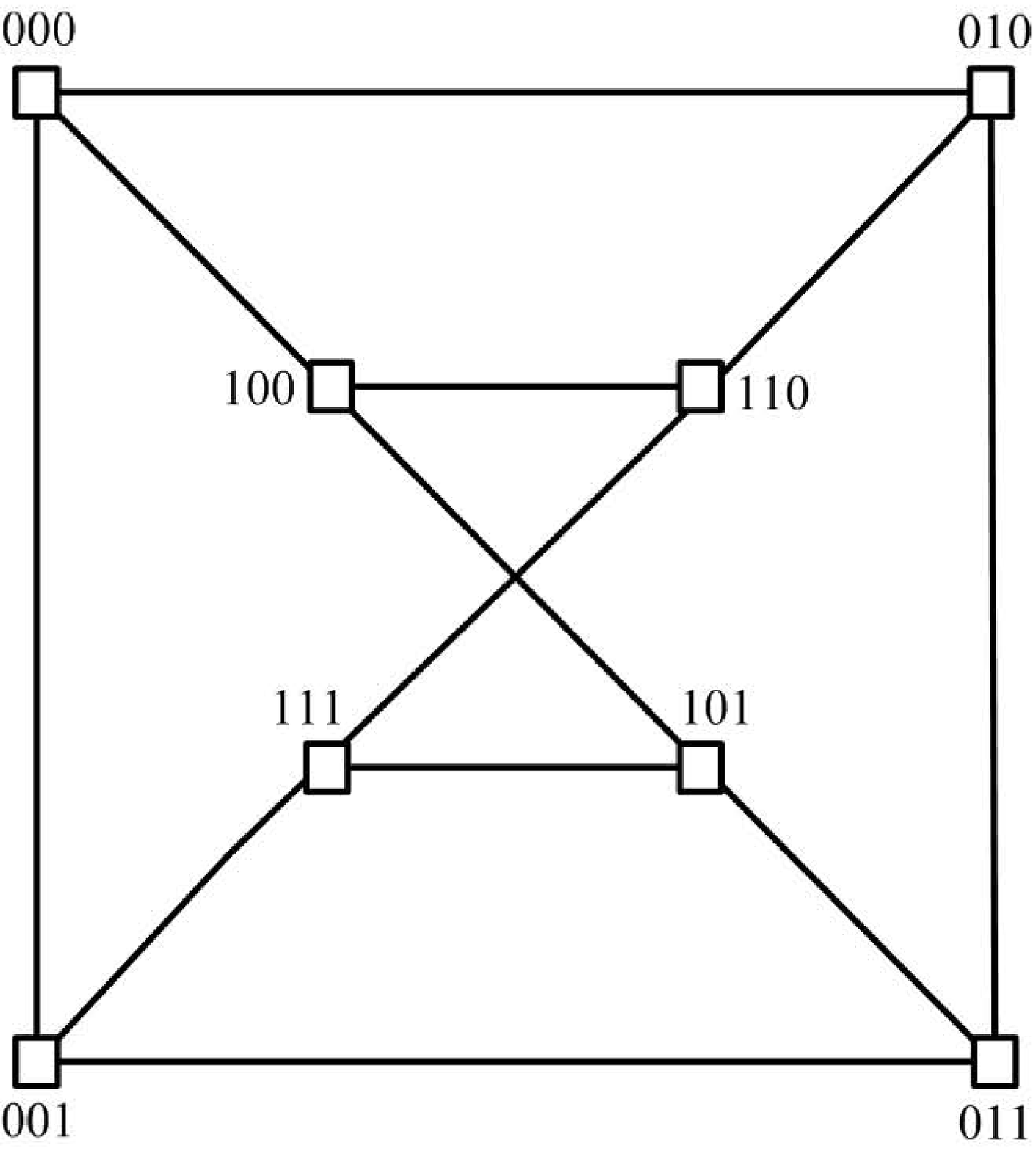}}
~~\subfigure[$B_{3}$.]{\includegraphics[totalheight=4cm, width=6.5cm]{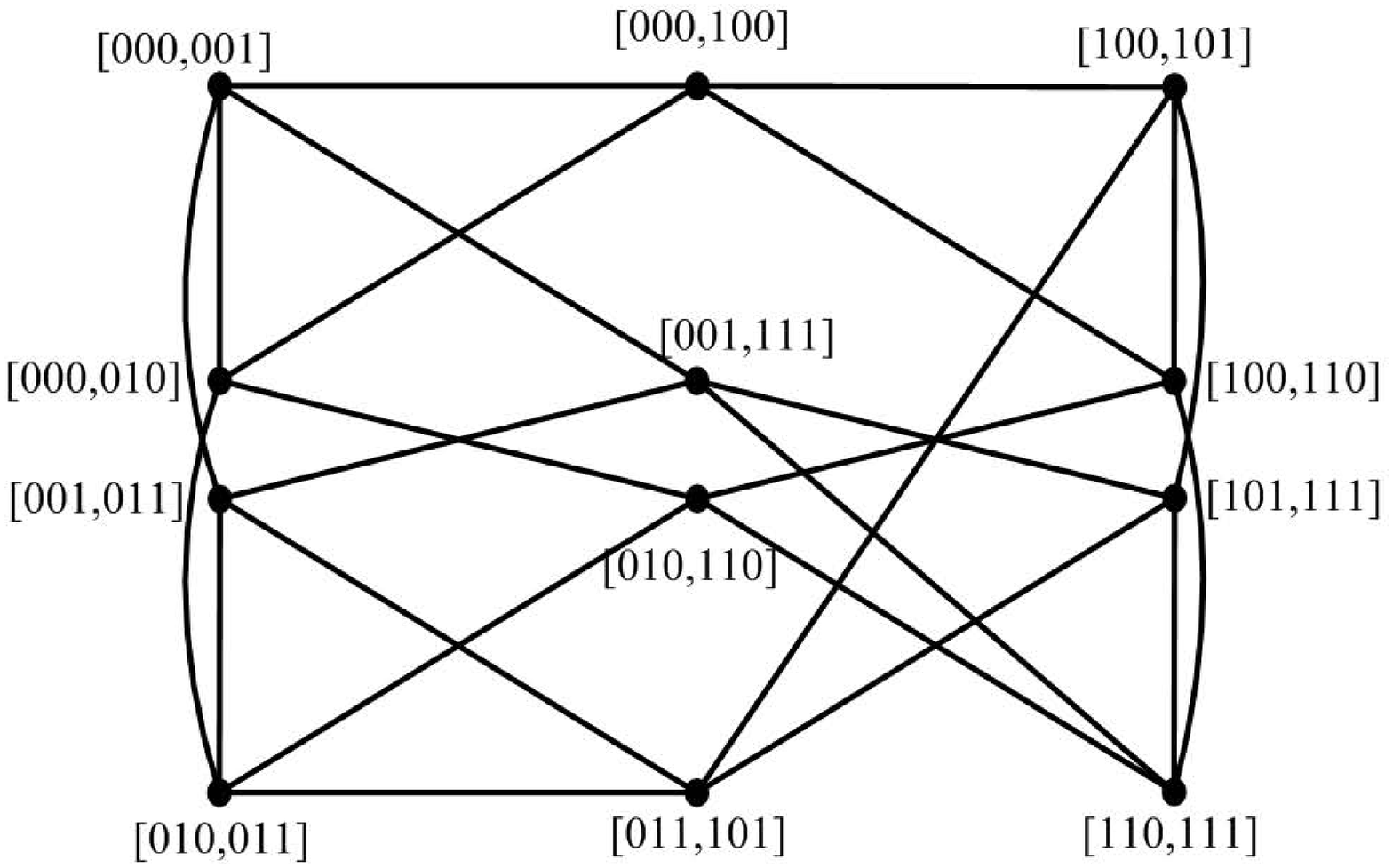}}\\
\caption{\label{} $A_{3}$, $CQ_{3}$ and $B_{3}$. }
\end{figure}

\begin{Lem} \emph{\cite{21}} $CQ_n$ is $n$-regular and triangle-free.
\end{Lem}
\begin{Lem} \emph{\cite{17}} For $B_n$, we have the statements below:
\begin{enumerate}[(i)]
\item \label{cond 1} $B_n$ is $(2n-2)$-regular and  has $n2^{n-1}$ vertices, $n(n-1)2^{n-1}$ edges.
\item \label{cond 2} $\kappa (B_n)=2n-2$.
\item \label{cond 3} $B_n$ is the line graph of $CQ_n$.
\end{enumerate}
\end{Lem}

\begin{Lem} \emph{\cite{17}} For $n\geq 3$ and $0\leq g\leq n-3$,
\begin{equation*}
\kappa_g(B_n)=\left\{
\begin{aligned}
4,~~if~n=3;\\
2n+g(n-2)-2,~~if~n\geq 4.\\
\end{aligned}
\right.
\end{equation*}
\end{Lem}

\begin{Lem}\label{TC} $B_n[N_{B_n}(u)]$ is two independent complete graphs $K_{n-1}$.
\end{Lem}
\begin{proof} Given a vertex $u=[v,w]$ in $B_n$, where $vw$ is an edge of $CQ_n$. Since $CQ_n$ is $n$-regular, we set $N_{CQ_n}(v)=\{w\}\cup \{v^i|0\leq i\leq n-2\}$ and $N_{CQ_n}(w)=\{v\}\cup \{w^i|0\leq i\leq n-2\}$. Then $N_{B_n}(u)=\{[ v, v^i], [w, w^i]|0\leq i\leq n-2\}$.  According to the Definition \ref{Bn},
$B_n[\{[v, v^i]|0\leq i\leq n-2\}]\cong K_{n-1}$ and $B_n[\{[w, w^i]|0\leq i\leq n-2\}]\cong K_{n-1}$. Suppose $[v, v^i]$ is adjacent to $[w, w^j]$ in $B_n$. Then $v^i=w^j$, which implies that $v, w$ and $v^i$ (or $w^j$) induce a triangle in $CQ_n$, a contradiction by Lemma 4.1. Thus $B_n[N_{B_n}(u)]$ is two independent complete graphs $K_{n-1}.$
\end{proof}
\begin{figure}[!htbp]
\centering
\includegraphics[height=6cm, width=15.5cm]{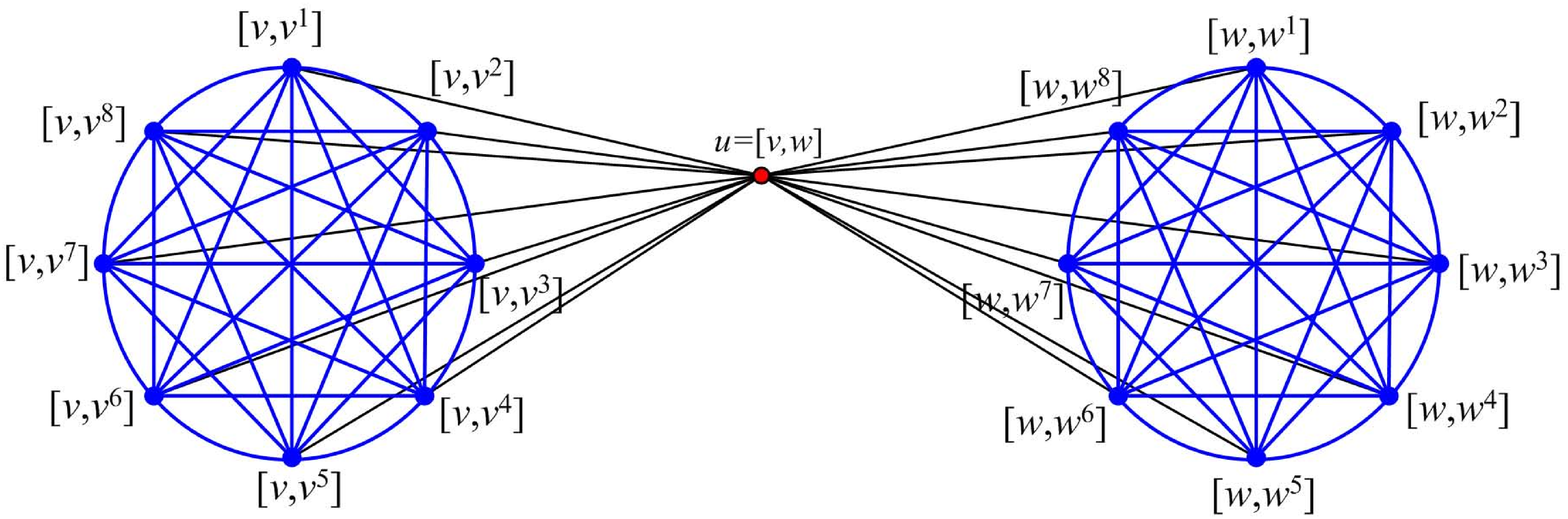}
\caption{\label{} Illustration for $B_9[N_{B_9}(u)]$. }
\end{figure}
\begin{The}
\begin{equation*}
\kappa (B_n; K_{1,1})=\kappa^s (B_n; K_{1,1})=\left\{
\begin{aligned}
n-1,~~for~odd~n\geq 5;\\
n,~~for~even~n\geq 4.\\
\end{aligned}
\right.
\end{equation*}
\end{The}
\begin{proof} Assume $\kappa^s (B_n; K_{1,1})\leq n-2$ ($\leq n-1$) for odd (even) $n$. Then $B_n$ has a $K_{1,1}$-substructure cut $\mathcal{F}$. Let $C$ be a smallest component of $B_n-\mathcal{F}$. For odd $n\geq 5$, $|V(\mathcal{F})|\leq 2(n-2)<\kappa (B_n)$ by Lemma 4.2, a contradiction. Thus $\kappa^s (B_n; K_{1,1})\geq n-1$ for odd $n\geq 5$.
For even $n\geq 4$, $|V(\mathcal{F})|\leq 2(n-1)<\kappa_1 (B_n)$ by Lemma 4.3. Then $|V(C)|=1$, say $u$. We have  $N_{B_n}(u)\subseteq V(\mathcal{F})$. By Lemma \ref{TC}, $n-1$ subgraphs of $K_{1,1}$ can not cover all vertices of $N_{B_n}(u)$, a contradiction. Thus $\kappa^s (B_n; K_{1,1})\geq n$ for even $n\geq 4$.

We give a $K_{1,1}$-structure cut of $B_n$ for $n\geq 4$. Let $u=[v, w]$ be a vertex of $B_n$ with $v=00\ldots 00$ and $w=10\ldots 00$. Then $N_{B_n}(u)=\{[v,v^i], [w, w^i]|0\leq i\leq n-2\}$. For $0\leq j\leq \lfloor\frac{n-2}{2}\rfloor-1$, we set
$S_j=\{[v,v^{2j}]; [v, v^{2j+1}]\}$ and $S'_j=\{[w,w^{2j}]; [w, w^{2j+1}]\}$. For odd $n$, we set
$$
\aligned
S_{\lfloor\frac{n-2}{2}\rfloor}=\{[v,v^{n-3}]; [v, v^{n-2}]\},~~S'_{\lfloor\frac{n-2}{2}\rfloor}=\{[w,w^{n-3}]; [w, w^{n-2}]\}.
\endaligned
$$
For even $n$, we set
$$
\aligned
S_{\lfloor\frac{n-2}{2}\rfloor}=\{[v,v^{n-2}]; [v^{n-2}, v^{n-2, n-1}]\},~~S'_{\lfloor\frac{n-2}{2}\rfloor}=\{[w,w^{n-2}]; [w^{n-2}, w^{n-2, n-1}]\}.
\endaligned
$$
Let $\mathcal{F}=\{S_j, S'_j|0\leq j\leq \lfloor\frac{n-2}{2}\rfloor\}$. We find $N_{B_n}(u)\subseteq V(\mathcal{F})$ and $[11\ldots 11, 11\ldots 10]\notin V(B_n-\mathcal{F})$, which implies that $B_n-\mathcal{F}$ has two components. Then $u$ is a singleton and $[11\ldots 11, 11\ldots 10]$ belongs to the other component. Thus $\kappa (B_n; K_{1,1})\leq n-1$ for odd $n\geq 5$ and $\kappa (B_n; K_{1,1})\leq n$ for even $n\geq 4$.
\end{proof}

\begin{Lem}\label{SL} For $n\geq 4$ and $2\leq t\leq 2n-3$, let $r$ be remainder of $n-1$ divided by $1+t$. Then
\begin{equation*}
\kappa^s(B_n; K_{1,t})\geq
\left\{
\begin{aligned}
&\frac{2n-4}{1+t}+1,~~if~ 2\leq t\leq n-3~and~r=1,\\
&2\lceil\frac{n-1}{1+t}\rceil,~~otherwise.\\
\end{aligned}
\right.
\end{equation*}\end{Lem}
\begin{proof} Suppose to the contrary that $\kappa^s(B_n; K_{1,t})\leq \frac{2n-4}{1+t}$ for $2\leq t\leq n-3$ and $r=1$, and $\kappa^s(B_n; K_{1,t})\leq 2\lceil\frac{n-1}{1+t}\rceil-1$ for otherwise.
Then $B_n$ has a $K_{1,t}$-substructure cut $\mathcal{F}$. Let $C$ be a smallest component of $B_n-\mathcal{F}$. Then $|V(\mathcal{F})|\geq \kappa(B_n)=2n-2$ by Lemma 4.2. For $2\leq t\leq n-3$, if $r=0$, then
$|V(\mathcal{F})|\leq (1+t)(2\lceil\frac{n-1}{1+t}\rceil-1)<2n-2$, a contradiction.
If $r=1$, then
$|V(\mathcal{F})|\leq (1+t)\frac{2n-4}{1+t}=2n-4<2n-2$,
 a contradiction. Similarly, for $n-2\leq t\leq 2n-4$, $|V(\mathcal{F})|\leq 1+t\leq 2n-3<2n-2$, a contradiction.
Then we consider the cases $t=2n-3$, or $2\leq t\leq n-3$ and $r\geq 2$. We find $|V(\mathcal{F})|\leq (1+t)(2\lceil\frac{n-1}{1+t}\rceil-1)=1+t\leq 2n-2$ for $t=2n-3$, and
$|V(\mathcal{F})|\leq (1+t)(2\lceil\frac{n-1}{1+t}\rceil-1)\leq (1+t)(2\frac{n-1+t}{1+t}-1)=2n-3+t\leq 3n-6$ for $2\leq t\leq n-3$ and $r\geq2$. By Lemma 4.3, $|V(\mathcal{F})|< \kappa_1(B_n)$.
Then $|V(C)|=1$, say $u=[v,w]$. Thus $N_{B_n}(u)\subseteq V(\mathcal{F})$, which implies that $B_n[N_{B_n}(u)]$ has a star $K_{1, 2n-3}$ for $t=2n-3$, contradicting Lemma \ref{TC}.
For $2\leq t\leq n-3$ and $r\geq 2$, $|\mathcal{F}|\leq 2\lceil\frac{n-1}{1+t}\rceil-1$. If all star-centers are in $N_{B_n}(u)$, then $2\lceil\frac{n-1}{1+t}\rceil-1$ stars can not cover all vertices of $N_{B_n}(u)$ by Lemma \ref{TC}. Thus there exists a star whose center $x$ is not in $N_{B_n}(u)\cup \{u\}$. We know that $x$ has at most two neighbors in $N_{B_n}(u)$ since $[v,v^i]$ and $[v,v^j]$ ($[w,w^i]$ and $[w,w^j]$) have no common neighbor in $V(B_n)-N_{B_n}(u)-\{u\}$. Then $V(\mathcal{F})\leq (1+t)\cdot2\cdot\lfloor\frac{n-1}{1+t}\rfloor+2= (1+t)\cdot2\cdot\frac{n-1-r}{1+t}+2=2n-2r<2n-2$ for $r\geq 2$. It is a contradiction.
\end{proof}

\begin{Lem}\label{SU} For $n\geq 4$ and $2\leq t\leq 2n-3$, let $r$ be remainder of $n-1$ divided by $1+t$. Then
\begin{equation*}
\kappa(B_n; K_{1,t})\leq
\left\{
\begin{aligned}
&\frac{2n-4}{1+t}+1,~~if~ 2\leq t\leq n-3~and~r=1,\\
&2\lceil\frac{n-1}{1+t}\rceil,~~otherwise.
\end{aligned}
\right.
\end{equation*}
\end{Lem}
\begin{proof} Let $u=[v, w]$ be a vertex of $B_n$ with $v=00\ldots 00$ and $w=10\ldots 00$. Then $N_{B_n}(u)=\{[v,v^i], [w, w^i]|0\leq i\leq n-2\}$.

For $n-2\leq t\leq 2n-3$, let $S_1$ ($S_2$) be a star with center $[v,v^0]$ ($[w, w^0]$) and the leaves are neighbors of $[v,v^0]$ ($[w, w^0]$). Precisely,
\begin{equation*}
\left\{
\begin{aligned}
&S_1=\{[v,v^0]; [v, v^i], t_j|1\leq i\leq n-2, n-1\leq j\leq t\},\\
&S_2=\{[w,w^0]; [w, w^i], t'_j|1\leq i\leq n-2, n-1\leq j\leq t\},\\
\end{aligned}
\right.
\end{equation*}
where $\{t_j|n-1\leq j\leq t\}\subseteq N_{B_n}([v, v^0])\setminus \{u,[v, v^i]|1\leq i\leq n-2\}$ and $\{t'_j|n-1\leq j\leq t\}\subseteq N_{B_n}([w, w^0])\setminus \{u,[w, w^i]|1\leq i\leq n-2\}$. Let $\mathcal{F}=\{S_1, S_2\}$. Since $N_{B_n}(u)\subseteq V(\mathcal{F})$ and $[11\ldots 11, 11\ldots 10]\notin V(B_n-\mathcal{F})$, $B_n-\mathcal{F}$ has two components. Then $u$ is a singleton and $[11\ldots 11, 11\ldots 10]$ belongs to  the other component.

For $2\leq t\leq n-3$ and
$1\leq i\leq \lfloor\frac{n-1}{1+t}\rfloor$, we set $S_i$ ($S'_i$) as a star with center $[v, v^{(i-1)(t+1)}]$ ($[w, w^{(i-1)(t+1)}]$) and the leaves are $[v, v^{(i-1)(t+1)+k}]$ ($[w, w^{(i-1)(t+1)+k}]$) for each $1\leq k\leq t$. The more detailed expressions are
\begin{equation*}
\left\{
\begin{aligned}
&S_i=\{[v, v^{(i-1)(t+1)}]; [v, v^{(i-1)(t+1)+k}]|1\leq k\leq t\}\\~~ 
&S'_i=\{[w, w^{(i-1)(t+1)}]; [w, w^{(i-1)(t+1)+k}]|1\leq k\leq t\}~~ 
\end{aligned}
\right.
\end{equation*}
If $r=1$, then we set $S''=\{[v^{n-2}, w^{n-2}]; [v, v^{n-2}], [w, w^{n-2}], t_i|1\leq i\leq t-2\}$
where $\{t_i|1\leq i\leq t-2\}\subseteq N_{B_n}([v^{n-2}, w^{n-2}])\setminus \{[v, v^{n-2}], [w, w^{n-2}]\}$.
If $r\geq 2$, then we set
\begin{equation*}
\left\{
\begin{aligned}
&S''_1=\{[v, v^{n-2}]; [v, v^i], t_j|n-r-1\leq i\leq n-3, 1\leq j\leq t-r+1\},\\
&S''_2=\{[w, w^{n-2}]; [w, w^i], t'_j|n-r-1\leq i\leq n-3, 1\leq j\leq t-r+1\},\\
\end{aligned}
\right.
\end{equation*}
where $\{t_j|1\leq j\leq t-r\}\subseteq N_{B_n}([v, v^{n-2}])\setminus \{u, [v, v^{i}]|0\leq i\leq n-3\}$
and $\{t'_j|1\leq j\leq t-r\}\subseteq N_{B_n}([w, w^{n-2}])\setminus \{u, [w, w^{i}]|0\leq i\leq n-3\}$.
Let $\mathcal{F}=\{S_i, S'_i|1\leq i\leq \lfloor\frac{n-1}{1+t}\rfloor\}$ for $r=0$, $\mathcal{F}=\{S_i, S'_i, S''|1\leq i\leq \lfloor\frac{n-1}{1+t}\rfloor\}$ for $r=1$ and $\mathcal{F}=\{S_i, S'_i, S''_1, S''_2|1\leq i\leq \lfloor\frac{n-1}{1+t}\rfloor\}$ for $r\geq 2$. We find $N_{B_n}(u)\subseteq V(\mathcal{F})$ and $[11\ldots 11, 11\ldots 10]\notin V(B_n-\mathcal{F})$, which implies that $B_n-\mathcal{F}$ has two components and $u$ is a singleton and $[11\ldots 11, 11\ldots 10]$ belongs to other component.
\end{proof}
By Lemmas \ref{SL} and \ref{SU}, we have the following result.
\begin{The}\label{S} For $n\geq 4$ and $2\leq t\leq 2n-3$, let $r$ be remainder of $n-1$ divided by $1+t$. Then
\begin{equation*}
\kappa(B_n; K_{1,t})=\kappa^s(B_n; K_{1,t})=
\left\{
\begin{aligned}
&\frac{2n-4}{1+t}+1,~~if~ 2\leq t\leq n-3~and~r=1,\\
&2\lceil\frac{n-1}{1+t}\rceil,~~otherwise.
\end{aligned}
\right.
\end{equation*}
\end{The}
\begin{Lem}\label{PL} For $n\geq 4$ and $4\leq k\leq 2n-1$, we have
\begin{equation*}
\kappa^s(B_n; P_{k})\geq
\left\{
\begin{aligned}
&\frac{2n-2}{k},~~if~ 4\leq k\leq n-1~and~k\mid n-1,\\
&\lceil\frac{2n-1}{k}\rceil,~~otherwise.\\
\end{aligned}
\right.
\end{equation*}
\end{Lem}
\begin{proof} Suppose $\kappa^s (B_n; P_{k})\leq \frac{2n-2}{k}-1$ for $4\leq k\leq n-1$ and $k\mid n-1$, and $\kappa^s (B_n; P_{k})\leq \lceil\frac{2n-1}{k}\rceil-1$ for otherwise. Then there exists a $P_k$-substructure cut $\mathcal{F}$. Let $C$ be a smallest component of $B_n-\mathcal{F}$. Then $|V(\mathcal{F})|\geq \kappa (B_n)=2n-2$ by Lemma 4.2.
For $4\leq k\leq n-1$ and $k\mid n-1$, $|V(\mathcal{F})|\leq k\cdot (\frac{2n-2}{k}-1)=2n-2-k<2n-2$, a contradiction. Similarly, for $n\leq k\leq 2n-3$, $|V(\mathcal{F})|\leq k\cdot (\lceil\frac{2n-1}{k}\rceil-1)=k\leq 2n-3<2n-2$, a contradiction.

We have $|V(C)|=1$, say $C=\{u\}$, for $k=2n-2$ or $4\leq k\leq n-1$ and $k\nmid n-1$ since $|V(\mathcal{F})|\leq k\cdot (\lceil\frac{2n-1}{k}\rceil-1)\leq k(\frac{2n-1+k-1}{k}-1)=2n-2<\kappa_1(B_n)$ by Lemma 4.3.
Then $N_{B_n}(u)\subseteq V(\mathcal {F})$, which implies that $B_n[N_{B_n}(u)]$ has a hamiltonian path $P_{2n-2}$ and $\frac{2n-2}{k}$ vertex-disjoint $P_{k}$'s for $4\leq k\leq n-1$ and $k\nmid n-1$. It is impossible by Lemma \ref{TC}.
\end{proof}

\begin{Lem}\label{PU} For $n\geq 4$ and $4\leq k\leq 2n-1$, we have
\begin{equation*}
\kappa(B_n; P_{k})\leq
\left\{
\begin{aligned}
&\frac{2n-2}{k},~~if~ 4\leq k\leq n-1~and~k\mid n-1,\\
&\lceil\frac{2n-1}{k}\rceil,~~otherwise.\\
\end{aligned}
\right.
\end{equation*}
\end{Lem}
\begin{proof} Let $u=[v, w]$ be a vertex of $B_n$ with $v=00\ldots 00$ and $w=10\ldots 00$. Then $N_{B_n}(u)=\{[v,v^i], [w, w^i]|0\leq i\leq n-2\}$. For convenience, we set $P_{n-1}(w)=\langle [w,w^0], [w, w^1],\\ \ldots, [w, w^{n-2}]\rangle$ as a path with $n-1$ vertices.

 For $k=2n-1$, we set
 $\mathcal{F}=P_{2n-1}=\langle P_{n-1}(v), [v^{n-2}, w^{n-2}], P^{-1}_{n-1}(w) \rangle$.
 For $n\leq k\leq 2n-2$, let $P^1_{2n-2}=\langle P_{n-1}(v), P_{n-1}(v^{n-2})-[v,v^{n-2}]\cup [v^{n-2},v^{n-2,n-1}]\rangle$ and $P^2_{2n-2}=\langle P_{n-1}(w), P_{n-1}(w^{n-2})-[w,w^{n-2}]\cup [w^{n-2}, w^{n-2,n-1}]\rangle$. Then we set $\mathcal{F}=\{P^1_{k}, P^2_{k}\}$,  where $P^1_{k}$ ($P^2_{k}$) along $P^1_{2n-2}$ ($P^2_{2n-2}$) with end-vertices $[v,v^0]$ ($[w,w^0]$) and the $k$-th vertex of $P^1_{2n-2}$ ($P^2_{2n-2}$).

 For $4\leq k\leq n-1$ and $k\mid n-1$, we set $\mathcal{F}=\{P^i_{k}|1\leq i\leq \frac{2n-2}{k}\}$, where $P^1_k$ along $P_{n-1}(v)$ with end-vertices $[v,v^0]$ and $[v,v^{k-1}]$, $P^2_k$ along $P_{n-1}(v)$ with end-vertices $[v,v^{k}]$ and $[v,v^{2k-1}]$, by the analogous, $P^{\frac{n-1}{k}}_k$ along $P_{n-1}(v)$ with end-vertices $[v,v^{n-k-1}]$ and $[v,v^{n-2}]$. Similarly, we set $P^{\frac{n-1}{k}+1}_k$ along $P_{n-1}(w)$ with end-vertices $[w,w^0]$ and $[w,w^{k-1}]$, by the analogous, $P^{2\frac{n-1}{k}}_k$ along $P_{n-1}(v)$ with end-vertices $[w,w^{n-k-1}]$ and $[w,w^{n-2}]$.
For $4\leq k\leq n-1$ and $k\nmid n-1$, let $P_{3n-2}=\langle P_{2n-1}, P_{n-1}(w^0)-[w,w^0]\cup [w^0, w^{0,n-1}]\rangle$ and $\mathcal{F}=\{P^i_{k}|1\leq i\leq \lceil\frac{2n-1}{k}\rceil\}$, where $P^1_k$ along $P_{3n-2}$ with end-vertices $[v,v^0]$ and $[v,v^{k-1}]$,   $P^2_k$ along $P_{3n-2}$ with end-vertices $[v,v^{k}]$ and $[v,v^{2k-1}]$, by the analogous, $P^{\lceil\frac{2n-1}{k}\rceil}_k$ along $P_{3n-2}$ with end-vertices are the $(k(\lceil\frac{2n-1}{k}\rceil-1)+1)$-th  and $k\lceil\frac{2n-1}{k}\rceil$-th vertices of $P_{3n-2}$.

 We find $N_{B_n}(u)\subseteq V(\mathcal{F})$ and $[11\ldots 11, 11\ldots 10]\notin V(B_n-\mathcal{F})$, which implies that $B_n-\mathcal{F}$ has two components. Then $u$ is a singleton and $[11\ldots 11, 11\ldots 10]$ belongs to other component.
\end{proof}

\begin{The}\label{P}For $n\geq 4$ and $4\leq k\leq 2n-1$, we have
\begin{equation*}
\kappa(B_n; P_{k})=\kappa^s(B_n; P_k)=
\left\{
\begin{aligned}
&\frac{2n-2}{k},~~if~ 4\leq k\leq n-1~and~k\mid n-1,\\
&\lceil\frac{2n-1}{k}\rceil,~~otherwise.\\
\end{aligned}
\right.
\end{equation*}
\end{The}
We have $\kappa^s(B_n; C_k)\leq \kappa^s(B_n; P_k)$. By the similar argument as Lemma \ref{PL}, we obtain the following result.
\begin{The}
For $n\geq 4$ and $4\leq k\leq 2n-1$, we have
\begin{equation*}
\kappa^s(B_n; C_k)=
\left\{
\begin{aligned}
&\frac{2n-2}{k},~~if~ 4\leq k\leq n-1~and~k\mid n-1,\\
&\lceil\frac{2n-1}{k}\rceil,~~otherwise.\\
\end{aligned}
\right.
\end{equation*}
\end{The}

\begin{Lem}\label{CL} For $n\geq 5$ and $3\leq k\leq 2n$, let $r$ be remainder of $n-1$ divided by $k$. Then
\begin{equation*}
\kappa(B_n; C_k)\geq
\left\{
\begin{aligned}
&2\lceil\frac{n-1}{k}\rceil-1,~~if~3\leq k\leq n-1~and~1\leq r\leq \lfloor\frac{k}{2}\rfloor-1,~or~ k=2n,\\
&3,~~if~k=n,\\
&2\lceil\frac{n-1}{k}\rceil,~~otherwise.\\
\end{aligned}
\right.
\end{equation*}
\end{Lem}
\begin{proof} Suppose to the contrary that $\kappa(B_n; C_k)\leq 2\lceil\frac{n-1}{k}\rceil-2$ for $3\leq k\leq n-1$ and $1\leq r\leq \lfloor\frac{k}{2}\rfloor-1$, or $k=2n$, $\kappa(B_n; C_k)\leq 2$ for $k=n$ and $\kappa(B_n; C_k)\leq 2\lceil\frac{n-1}{k}\rceil-1$ for otherwise. Then there exists a $C_k$-structure cut $\mathcal{F}$. Let $C$ be a smallest component. We have $|V(\mathcal{F})|\geq \kappa (B_n)=2n-2$ by Lemma 4.2.
 For $3\leq k\leq n-1$ and $1\leq r\leq \lfloor\frac{k}{2}\rfloor-1$, or $k=2n$,
$|V(\mathcal{F})|\leq k(2\lceil\frac{n-1}{k}\rceil-2)\leq k(\frac{2(n+k-2)}{k}-2)=2n-4<2n-2$, a contradiction. We have a Similar contradiction for $3\leq k\leq n-1$ and $r=0$ since $|V(\mathcal{F})|<2n-2$.

For $n\leq k\leq 2n-1$, it is easy to find $|V(\mathcal{F})|\leq 2n$. For $3\leq k\leq n-1$ and $\lfloor\frac{k}{2}\rfloor\leq r\leq k-1$, we have
$$
\aligned
|V(\mathcal{F})|\leq k\cdot(2\lceil\frac{n-1}{k}\rceil-1)\leq k\cdot(2\frac{n-1+\lceil\frac{k}{2}\rceil}{k}-1)\leq2(n-1+\frac{k+1}{2})-k=2n-1.
\endaligned
$$
Then $|V(\mathcal{F})|<\kappa_1(B_n)$ for $n\geq 5$ by Lemma 4.3. Thus $|V(C)|=1$, say $u=[v,w]$, which implies that
$N_{B_n}(u)\subseteq V(\mathcal{F})$ for $n\leq k\leq 2n-1$, or $3\leq k\leq n-1$ and $\lfloor\frac{k}{2}\rfloor\leq r\leq k-1$. We find $|\mathcal{F}|\leq 1$ for $n+1\leq k\leq 2n-1$, that is, a $C_k$ covers all vertices of $N_{B_n}(u)$, contradicting Lemma \ref{TC}. We find $|\mathcal{F}|\leq 2$ for $n=k$, which implies that we need two $C_n$'s to cover all vertices of $N_{B_n}(u)$. We know the longest cycle is $C_{n-1}$ in $B_n[N_{B_n}(u)]$ by Lemma \ref{TC}.  Then a $C_n$ cover $n-1$ vertices of $N_{B_n}(u)$ and these $n-1$ vertices are in a $K_{n-1}$. Let $x=[x_1, x_2]$ be a vertex of $C_n$ but $x\notin N_{B_n}(u)$. Then $x$ is adjacent to both $[v,v^i]$ and $[v,v^j]$ (or $[w,w^i]$ and $[w,w^j]$) for $i,j\in[n]$. Thus $x_1=v^i$ (or $w^i$) and $x_2=v^j$ (or $w^j$). It is a contradiction since $[v^i,v^j]\notin V(B_n)$ and $[w^i,w^j]\notin V(B_n)$.
For $3\leq k\leq n-1$ and $\lfloor\frac{k}{2}\rfloor\leq r\leq k-1$, we have $2r\geq k-1$ and these $2r$ vertices induce two disjoint $P_r$'s. Then we know a $C_k$ can not cover two disjoint $P_r$'s by Lemma \ref{TC}, a contradiction.
\end{proof}

\begin{Lem}\label{CU} For $n\geq 5$ and $6\leq k\leq 2n$, let $r$ be remainder of $n-1$ divided by $k$. Then
\begin{equation*}
\kappa(B_n; C_k)\leq
\left\{
\begin{aligned}
&2\lceil\frac{n-1}{k}\rceil-1,~~if~6\leq k\leq n-1~and~1\leq r\leq \lfloor\frac{k}{2}\rfloor-1,~or~ k=2n,\\
&3,~~if~k=n,\\
&2\lceil\frac{n-1}{k}\rceil,~~otherwise.\\
\end{aligned}
\right.
\end{equation*}
\end{Lem}
\begin{proof} Let $u=[v, w]$ be a vertex of $B_n$ with $v=00\ldots 00$ and $w=10\ldots 00$. Then $N_{B_n}(u)=\{[v,v^i], [w, w^i]|0\leq i\leq n-2\}$. Let $P_{n-1}(v)=\langle [v,v^1], [v, v^0], [v, v^2], [v,v^3], \ldots, [v,v^{n-2}]\rangle$ be a path with $n-1$ vertices.

For $k=2n$, we set
$$
\aligned
\mathcal{F}=C_{k}=\langle [v^1, w^1], P_{n-1}(v), [v^{n-2}, w^{n-2}], P^{-1}_{n-1}(w)\rangle.
\endaligned
$$

For $n+1\leq k\leq 2n-1$, we set $\mathcal{F}=\{C^1_k, C^2_k\}$ and
$$
\aligned
C^1_k=\langle P_{n-1}(v), [v^{n-2}, v^{n-2,1}], [v^{n-2,1}, v^1], t_1, \ldots, t_{k-n-1}\rangle
\endaligned
$$
where $\{t_i|1\leq i\leq k-n-1\}\subseteq \{[v^1, v^{1, i}]|i=0, n-1$ and $2\leq i\leq n-3\}$,
$$
\aligned
C^2_k=\langle P_{n-1}(w), [w^{n-2}, w^{n-2,1}], [w^{n-2, 1}, w^1], t_1, \ldots, t_{k-n-1}\rangle
\endaligned
$$
where $\{t_i|1\leq i\leq k-n-1\}\subseteq \{[w^1, w^{1, i}]|i=0, n-1$ and $2\leq i\leq n-3\}$.

For $k=n$, we set $\mathcal{F}=\{C^1_k, C^2_k, C^3_k\}$ and
$$
\aligned
C^1_k=\langle P_{n-1}(v)-[v,v^{n-2}], [v^{n-3}, v^{n-3,1}], [v^{n-3, 1}, v^1]\rangle;
\endaligned
$$
$$
\aligned
C^2_k=\langle P_{n-1}(w)-[w,w^{n-2}], [w^{n-3}, w^{n-3, 1}], [w^{n-3, 1}, w^1]\rangle;
\endaligned
$$
$$
\aligned
C^3_k=\langle [v,v^{n-2}], [v^{n-2}, w^{n-2}], [w,w^{n-2}], [w,w^1], [v^1, w^1], [v, v^1], t_1,\ldots, t_{k-6} \rangle,
\endaligned
$$
where $\{t_i|1\leq i\leq k-6\}\subseteq \{[v,v^i]|2\leq i\leq n-2\}$.

For $6\leq k\leq n-1$ and $1\leq i\leq \lfloor \frac{n-1}{k}\rfloor$, we set
$$
\aligned
C^i_k=\langle [v,v^{(i-1)k}], [v,v^{(i-1)k+1}], \ldots, [v,v^{ik-1}]\rangle;
\endaligned
$$
$$
\aligned
C'^i_k=\langle [w,w^{(i-1)k}], [w,w^{(i-1)k+1}], \ldots, [w,w^{ik-1}]\rangle.
\endaligned
$$
If $r=0$, then we set $\mathcal{F}=\{C^i_k, C'^i_k|1\leq i\leq \lfloor\frac{n-1}{k}\rfloor\}$.
If $1\leq r\leq \lfloor\frac{k}{2}\rfloor-1$, then we set $\mathcal{F}=\{C^{\lceil\frac{n-1}{k}\rceil}_k\}\cup\{C^i_k, C'^i_k|1\leq i\leq \lfloor\frac{n-1}{k}\rfloor\}$ and
$$
\aligned
C^{\lceil\frac{n-1}{k}\rceil}_k&=\langle [w,w^1], [v^1, w^1], [v, v^1], t_1,\ldots, t_{k-2r-4}, [v,v^{n-r-1}], [v,v^{n-r}], \ldots, [v,v^{n-2}], \\ &[v^{n-2}, w^{n-2}], [w,w^{n-2}],[w,w^{n-3}], \ldots, [w,w^{n-r-1}],  [w,w^1]\rangle;
\endaligned
$$
where $\{t_i|1\leq i\leq k-2r-4\}\subseteq \{[v, v^{i}]|2\leq i\leq n-r-2\}$.
If $\lfloor\frac{k}{2}\rfloor\leq r< k-1$, then we set $\mathcal{F}=\{C^i_k, C'^i_k|1\leq i\leq \lceil\frac{n-1}{k}\rceil\}$ and
$$
\aligned
C^{\lceil\frac{n-1}{k}\rceil}_k=\langle [v,v^1], [v,v^{n-r-1}], [v,v^{n-r}], \ldots, [v,v^{n-2}],t_1,\ldots, t_{k-r-3}, [v^{n-2}, v^{n-2, 1}], [v^{n-2, 1}, v^1]\rangle,
\endaligned
$$ where $\{t_i|1\leq i\leq k-r-3\}\subseteq \{[v^{n-2}, v^{n-2, i}]|0\leq i\leq n-3\}$,
$$
\aligned
C'^{\lceil\frac{n-1}{k}\rceil}_k&=\langle [w,w^1], [w,w^{n-r-1}], [w,w^{n-r}], \ldots, [w,w^{n-2}],t_1,\ldots, t_{k-r-3}, [w^{n-2}, w^{n-2, 1}],\\& [w^{n-2, 1}, w^1]\rangle,
\endaligned
$$ where $\{t_i|1\leq i\leq k-r-3\}\subseteq \{[w^{n-2}, w^{n-2, i}]|1\leq i\leq n-3\}$.
If $r=k-1$, then we set $\mathcal{F}=\{C^i_k, C'^i_k|1\leq i\leq \lceil\frac{n-1}{k}\rceil\}$ and
$$
\aligned
C^{\lceil\frac{n-1}{k}\rceil}_k=\langle [v,v^0], [v,v^{n-k}], [v,v^{n-k+1}], \ldots, [v,v^{n-2}]\rangle,
\endaligned
$$
$$
\aligned
C'^{\lceil\frac{n-1}{k}\rceil}_k=\langle [w,w^0], [w,w^{n-k}], [w,w^{n-k+1}], \ldots, [w,w^{n-2}]\rangle.
\endaligned
$$

 We find $N_{B_n}(u)\subseteq V(\mathcal{F})$ and $[11\ldots 11, 11\ldots 10]\notin V(B_n-\mathcal{F})$, which implies that $B_n-\mathcal{F}$ has two components. Then $u$ is a singleton and $[11\ldots 11, 11\ldots 10]$ belongs to the other component.
\end{proof}
By Lemmas \ref{CL} and \ref{CU}, we have the following theorem.
\begin{The}\label{C}For $n\geq 5$ and $6\leq k\leq 2n$, let $r$ be remainder of $n-1$ divided by $k$. Then
\begin{equation*}
\kappa(B_n; C_k)=
\left\{
\begin{aligned}
&2\lceil\frac{n-1}{k}\rceil-1,~~if~6\leq k\leq n-1~and~1\leq r\leq \lfloor\frac{k}{2}\rfloor-1,~or~ k=2n,\\
&3,~~if~k=n,\\
&2\lceil\frac{n-1}{k}\rceil,~~otherwise.\\
\end{aligned}
\right.
\end{equation*}
\end{The}
\section{Conclusion }
In this paper, we obtained structure connectivities of two famous data center networks DCell and BCDC on some common structures.
For DCell network $D_{m,n}$, $m\geq 0$ and $n\geq 2$, we got that
$\kappa (D_{m,n}; K_{1,t})=\kappa^s (D_{m,n}; K_{1,t})=\lceil \frac{n-1}{1+t}\rceil+m$ for $1\leq t\leq m+n-2$ and $\kappa (D_{m,n}; K_s)= \lceil\frac{n-1}{s}\rceil+m$ for $3\leq s\leq n-1$ by analyzing the structural properties of $D_{m,n}$.
For BCDC network $B_n$, $n\geq 5$, we used the existing results of $g$-extra connectivity to obtain $\kappa(B_n; K_{1,t})$ and $\kappa^s(B_n; K_{1,t})$ for $1\leq t\leq 2n-3$ (see Theorems 4.5 and \ref{S}) and $\kappa(B_n; P_{k})$ and $\kappa^s(B_n; P_{k})$ for $4\leq k\leq 2n-1$ (see Theorem \ref{P}); $\kappa(B_n; C_{k})$ and $\kappa^s(B_n; C_{k})$ for $6\leq k\leq 2n$ (see Theorem \ref{C}). It is easy to find that $\kappa (D_{m,n}; P_{k})$ ($\kappa^s (D_{m,n}; P_{k})$) and $\kappa (D_{m,n}; C_{k})$ ($\kappa^s (D_{m,n}; C_{k})$) have no relevant conclusions.
Except for $K_{1,t}$, $P_{k}$, $C_{k}$ and $K_{s}$, there are still other structure connectivities of $D_{m,n}$ and $B_n$ have not been studied. So we will keep working on structure connectivity of data center networks.

\section*{Acknowledgement}
The work is supported by NSFC (Grant No. 11871256).

\end{document}